\documentclass{article}
\usepackage{amsfonts}
\newtheorem{theorem}{Theorem}[section]

\newtheorem{lemma}[theorem]{Lemma}

\def \squareforqed{\hbox{\rlap{$\sqcap$}$\sqcup$}}
\def \qed{\ifmmode\squareforqed\else{\unskip\nobreak\hfil
          \penalty50\hskip1em\null\nobreak\hfil\squareforqed
          \parfillskip=0pt\finalhyphendemerits=0\endgraf}\fi}
\newenvironment{proof}{\noindent{\bf Proof} }{\qed\medskip}

\def\FF{{\mathbb F}}
\def\ZZ{{\mathbb Z}}
\def\NN{{\mathbb N}}
\def\QQ{{\mathbb Q}}
\def\RR{{\mathbb R}}
\def\CC{{\mathbb C}}
\def\height{\hbox{\rm ht}}
\def\lc{\hbox{lc}}
\def\tc{\hbox{tc}}
\def\RB{\hbox{RB}}
\newcommand\eg{{\it e.g.}}
\newcommand\ie{{\it i.e.}}

\title{Bounds on Factors in $\ZZ[x]$}

\author{John Abbott}
\begin{document}
\maketitle

\begin{abstract}
  We gather together several bounds on the sizes of coefficients which can
  appear in factors of polynomials in $\ZZ[x]$; we include a new bound
  which was latent in a paper by Mignotte, and a few minor improvements to
  some existing bounds.  We compare these bounds and show that none is
  universally better than the others.  In the second part of the paper we
  give several concrete examples of factorizations where the factors have
  ``unexpectedly'' large coefficients.  These examples help us understand
  why the bounds must be larger than you might expect, and greatly extend
  the collection published by Collins.
\end{abstract}

\section{Introduction}

How large can the coefficients of a factor of a polynomial be?  Let's try
an example.  The polynomial $f$ given below has an irreducible
factorization of the form $f(x) = g(x) \, g(-x)$ in $\ZZ[x]$.  The largest
coefficient in $f$ is~$2$; but how big is the largest coefficient of $g$?
\begin{eqnarray*}
f &=& x^{80} - 2x^{78} + x^{76} + 2x^{74} + 2x^{70} + x^{68} + 2x^{66} + x^{64} + x^{62} + 2x^{60} + 2x^{58}  \cr
  & & - 2x^{54} + 2x^{52} + 2x^{50} + 2x^{48} - x^{44} - x^{36} + 2x^{32} + 2x^{30} + 2x^{28} - 2x^{26}  \cr
  & & + 2x^{22} + 2x^{20} + x^{18} + x^{16} + 2x^{14} + x^{12} + 2x^{10} + 2x^{6} + x^{4} - 2x^{2} + 1
\end{eqnarray*}

\medskip

The effective factorization of polynomials in $\ZZ[x]$ (and thus also in
$\QQ[x]$ by Gauss's lemma) is one of computer algebra's success stories.  Modern
implementations running on current hardware take only a few seconds to
factorize even quite large polynomials with degrees in the hundreds~---~a
feat which would have been utterly impossible before the advent of computer
algebra.  Being able to compute reasonable bounds on the sizes of factors is a
crucial part of this success.

Let us see why these bounds are so important.  About forty years ago the
ideas and algorithms of Berlekamp~\cite{Ber67} (and later~\cite{Ber70}) and
Zassenhaus~\cite{Zas69} finally made polynomial factorization
feasible.\footnote{The much older algorithm of Newton was feasible only for
  very small inputs.}  Now their essence lies at the heart of every general
implementation.  This Berlekamp--Zassenhaus scheme has maintained its
ubiquity, despite the appearance of a number of relatively recent results
of considerable theoretical importance (\eg~showing that polynomial
factorization is polynomial time), simply because it works so well in
practice.  A good survey of the development of factorization algorithms can
be found in~\cite{Kal82}, \cite{Kal90}, and~\cite{Kal92}.

\medskip

We recall briefly the principal steps in the Berlekamp--Zassenhaus
algorithm to factorize a polynomial $f \in \ZZ[x]$:
\begin{quote}
{\small
\begin{description}
\item{(1)} make~$f$ primitive and square-free
\item{(2)} pick a suitable prime $p$
\item{(3)} determine the factorization in $\FF_p[x]$
\item{(4)} lift to a factorization modulo a large enough power $p^k$
\item{(5)} recover the true factors in $\ZZ[x]$
\end{description}
}
\end{quote}
In this article we look at the answer to the question: {\it What value
  of $k$ should we use in step~$(4)$?}  In fact, this is just the same
as asking: {\it How large can the coefficients of an irreducible
  factor of $f \in \ZZ[x]$ be?}  Intuitively we would expect the
factors to be ``smaller'' than $f$: of course, the degree is smaller,
but the coefficients need not be (as we shall in
sections~\ref{one-large-factor} and~\ref{high-ratio}).

As we shall see below, the need for bounds on the factors is
exemplified by the polynomial $x^4+1$.  It is irreducible in $\ZZ[x]$
but factorizes in every $\FF_p[x]$ into either four linears or two
quadratics.  This phenomenon is quite general: it is well known that
a general polynomial in $\ZZ[x]$ often has a complete modular
factorization in $\FF_p[x]$ which is finer than that in $\ZZ[x]$.  This
knowledge is commonly exploited in a process called {\bf degree
  analysis}: we compute factorizations modulo several different
primes, and use the degrees of the modular factors to deduce possible
degrees for the true factors in $\ZZ[x]$.  In some cases degree
analysis alone can prove that a polynomial is irreducible.

In the case of $x^4+1$ degree analysis by itself cannot prove irreducibility,
as all modular factorizations are compatible with the existence of two
quadratic factors in $\ZZ[x]$.  In fact, we know that $x^4+1$ is
irreducible, so regardless of the chosen prime $p$, and for every exponent
$k$, the lifted factorization modulo $p^k$ contains ``false'' factors.
Now, if we had no bound on the sizes of the coefficients in the factors, we
would have to continue lifting indefinitely because we could never be sure
whether a few more lifting steps might produce the true factors in
$\ZZ[x]$.  In other words, without a bound on the factors, the ideas of
Berlekamp and Zassenhaus produce a method which would fail to terminate on
some inputs.

Thus good computable bounds for the sizes of coefficients in factors are
essential for making the Berlekamp-Zassenhaus method into a general
algorithm.  Interestingly, Newton's method does not need any such bounds.

\medskip

In section~\ref{notation} we set out the notation used in this paper.  In
section~\ref{deg-aware-bounds} we present the classical ``degree aware''
bounds for the coefficients of factors, and make improvements to two of
them.  Using the often overlooked fact that each of these bounds gives
limits for each individual coefficient of the factor, we show how they can
be combined together to produce a tighter result than any one in isolation.
We compare the various bounds using concrete examples.  In
section~\ref{single-factor-bounds} we present a new ``single factor'' bound
(which was latent in an article of Mignotte~\cite{Mig74}) together with the
bound from~\cite{BTW93}.  Again we use concrete examples in the comparison.
In section~\ref{single-vs-deg-aware} we confront the two types of bound:
those from section~\ref{deg-aware-bounds} and those
from section~\ref{single-factor-bounds}.  In light of the examples in our
comparisons, we conclude that no bound is universally better or worse than
the others.

One common feature of all these bounds is that they are almost always
surprisingly large~---~one would be tempted to say {\it unreasonably
  large.}  In section~\ref{one-large-factor} we exhibit some examples
which help us understand why the bounds in
section~\ref{deg-aware-bounds} must be so large; and in
section~\ref{high-ratio} there are further examples which help explain
why the bounds in section~\ref{single-factor-bounds} have to be large.
These examples are the fruits of extensive computations using {\it ad
  hoc\/} C++ programs assisted by CoCoA or CoCoALib~\cite{CoCoA}.  The
examples extend considerably the collection published by
Collins~\cite{Col04}.

\medskip
\noindent
{\small
{\bf Note}
The answer to the question right at the start is: the largest coefficient
of $g$ is~$36$.  This is the irreducible factorization with largest ratio that we have found.
}
%%%%%%%%%%%%%%%%%%%%%%%%%%%%%%%%%%%%%%%%%%%%%%%%%%%%%%%%%%%%%%%%%%%%%%%%%%%%%
%%%%%%%%%%%%%%%%%%%%%%%%%%%%%%%%%%%%%%%%%%%%%%%%%%%%%%%%%%%%%%%%%%%%%%%%%%%%%

\section{Notation}\label{notation}

We introduce the notation and terminology we shall use throughout the paper.

Let $f \in \ZZ[x]$ be a polynomial.  We say that $f$ is {\bf primitive} if
there is no common factor greater than~$1$ dividing its coefficients.  We
shall write $\lc(f)$ to denote the {\bf leading coefficient} of~$f$, and
$\tc(f)$ for its {\bf trailing coefficient}: \ie~$\tc(f) = f(0)$.  Define
the {\bf reverse} of a polynomial: $\bar{f}(x) = x^d f(\frac{1}{x})$
where~$d$ is the degree of~$f$.  We generalize slightly the usual notion of
palindromic: we say that~$f$ is $\pm$-{\bf palindromic} if $f(x) = \pm
\bar{f}(x)$.  We define a {\bf $*$-symmetric factorization} to be one of
the form $f(x) = g(x) \cdot g^*(x)$ where $g^*(x) = \bar{g}(-x)$; here is
an example
\begin{eqnarray*}
f &=& 12x^8 + 2x^7 + 4x^6 - 8x^5 + 5x^4 + 8x^3 + 4x^2 - 2x + 12 \cr
  &=& (3x^4 + 8x^3 + 12x^2 + 10x + 4)(4x^4 - 10x^3 + 12x^2 - 8x + 3)
\end{eqnarray*}

We define the {\bf height} of a polynomial to be the largest absolute value
of a coefficient: \ie~if $f = \sum a_i x^i$ then $\height(f) = \max \{
|a_i| \}$ where, as usual, $|z|$ denotes the absolute value of the complex
number $z \in \CC$.  The height will be our main measure of the size of a
polynomial.  Define the $l_p$-norm of the coefficients of $f = \sum a_i x^i
\in \CC[x]$ to be $|f|_p = \bigl( \sum |a_i|^p \bigr)^{1/p}$.  While we
will occasionally use the $l_2$-norm, we are most interested in the case
$p=\infty$ because $\height(f) = |f|_\infty$.

We define the {\bf ratio} of a factorization $f = g_1 g_2 \cdots g_s \in \CC[x]$ to be
$$
{1 \over \height(f)} \min \{ \height(g_1), \height(g_2),\ldots,\height(g_s) \}
$$
The ratio measures how large the factors are compared to their product.  We
shall be particularly interested in factorizations in $\ZZ[x]$ with a ratio
greater than~$1$.  We concentrate primarily on the case $s=2$, and will
look at the case $s>2$ in section~\ref{many-large-irred-factors}.

We recall {\bf Mahler's measure} $M(f) = |\lc(f)| \prod \max(1, |\alpha_i|)$ the
product being taken over all the complex roots of $f$.  Observe that
Mahler's measure is invariant under reversal, \ie~$M(\bar f) = M(f)$.
Actually computing the value of $M(f)$ seems to be difficult, but a good
approximation can be calculated using the ideas in~\cite{CMP87} and~\cite{DM90}.  In
particular, Mignotte~\cite{Mig74} proved that we can always use $|f|_2$ as an
upper bound for $M(f)$.

%%%%%%%%%%%%%%%%%%%%%%%%%%%%%%%%%%%%%%%%%%%%%%%%%%%%%%%%%%%%%%%%%%%%%%%%%%%%%
%%%%%%%%%%%%%%%%%%%%%%%%%%%%%%%%%%%%%%%%%%%%%%%%%%%%%%%%%%%%%%%%%%%%%%%%%%%%%

\section{Degree Aware Factor Coefficient Bounds}
\label{deg-aware-bounds}

In the Introduction we mentioned how important factor bounds are, so it
comes as no surprise to learn that it is a topic which has already
attracted some attention.  We distinguish two types of bound: the {\it
  degree aware bounds\/} which make use of information (\eg~obtained from
degree analysis) about the possible degrees of factors, and the newer
{\it single factor bounds\/} which apply to at least one factor (not
necessarily irreducible).  It seems that the single factor bounds cannot
exploit knowledge about possible degrees of factors.  In contrast, all
the degree aware bounds are increasing with the degree (at least up to
$\frac{1}{2} \deg(f)$), so each of them can be used as a single factor bound
just by computing the bounds for a factor of degree $\frac{1}{2} \deg(f)$.
We shall look at the single factor bounds in the next section.  In this
section we recall and compare several of the degree aware bounds which have
appeared in the literature.  We also present a minor improvement to two
of the bounds.

\medskip

In this section we shall assume we are given $f \in \ZZ[x]$ and also
$\delta \in \NN$, and the aim is to bound the coefficients of any factor $g
\in \ZZ[x]$ whose degree is at most $\delta$.  The four methods we present
actually produce bounds for the magnitudes of the coefficients of any
factor $g \in \CC[x]$ satisfying a natural scaling hypothesis (described
below).  Previous uses of these bounds generally used each method to
produce just a single number, namely a height bound valid for the whole
of~$g$~---~such an overall height bound is all that is needed to determine
how far one must lift.  In fact each of the bounding methods gives
individual limits for the separate coefficients of~$g$, and we shall
exploit this to improve the binomial bound (sect.~\ref{binomial-bound}) and
the Knuth--Cohen bound (sect.~\ref{knuth-cohen-bound}).  We give examples
to show that none of the bounds is universally superior, \ie~for each of
the bounding methods there are cases where it gives a lower overall height
bound than the others.  In section~\ref{combined-is-best} we shall combine
the individual coefficient bounds to obtain a result better than any one of
the bounds in isolation.

\subsubsection*{The Scaling Hypothesis}

Ideally we want bounds valid only for irreducible factors in $\ZZ[x]$, but
at the moment the only bounds known to us are valid for a much wider class
of factors, namely factors in $\CC[x]$ which are suitably scaled to avoid
problems with scalar factors.  So, even though we are primarily interested
in factorizations in $\ZZ[x]$, we shall be considering polynomials with
complex coefficients in this section.  So, let
$$
f=\sum_{i=0}^d a_i x^i \in \CC[x]
$$
be the polynomial of degree~$d$ whose factors we wish to bound.  For
convenience we shall assume that $a_0 \neq 0$, \ie~$\tc(f) \neq 0$.  The
factor whose coefficients we wish to bound is
$$
g = \sum_{i=0}^\delta b_i x^i \in \CC[x]
$$
We assume the following {\bf scaling hypothesis}: we require that the
factor~$g$ satisfy both $|\lc(g)| \le |\lc(f)|$ and $|\tc(g)| \le
|\tc(f)|$; note that this hypothesis is automatically satisfied if~$f$
and~$g$ are images in $\CC[x]$ of a polynomial and one of its factors in
$\ZZ[x]$.

\subsubsection*{The Reversal Trick}

We now make a simple observation which allows us to improve two of the
bounds below.  Polynomial multiplication and reversal commute: \ie~if
$f=g_1 g_2$ then $\bar{f} = \bar{g}_1 \bar{g}_2$.  Thus a bound for the
coefficient of $x^{\delta -k}$ in a degree $\delta$ factor of $\bar{f}$ is
also valid as a bound for the coefficient of $x^k$ in a factor of $f$.  So
the idea is simply to compute two sets of individual coefficient bounds:
one for a degree $\delta$ factor of $f$, and the other for a degree
$\delta$ factor of $\bar{f}$.  Then we combine them to get potentially
improved bounds for the coefficient of $x^k$ in a degree $\delta$ factor of
$f$.  Curiously, this simple idea does not seem to have been described
before even though it often produces usefully lower height bounds.

\subsection{The Binomial Bound}
\label{binomial-bound}

We shall use the binomial expansion to bound the coefficients of~$g$.  Here
is the basic idea.  Suppose we know a value $\rho$ which is an upper bound
for the magnitude of any complex root of~$f$.  Then we see that the $|b_i|$
are dominated by the corresponding coefficients of the polynomial $|\lc(g)|
(x+\rho)^\delta$.  We do not know what $|\lc(g)|$ is, but by the scaling
hypothesis we have $|\lc(g)| \le |\lc(f)|$.  Thus the $|b_i|$ are surely
dominated by the coefficients of $|\lc(f)|(x+\rho)^\delta$.

We can improve the binomial bound by using the reversal trick.  In detail,
let $\bar \rho$ be a root bound for $\bar f$ then by reasoning analogous
that above we see that $|b_i|$ is dominated by the smaller of the
coefficients of $x^i$ in $|\lc(f)| (x+\rho)^\delta$ and in $|\tc(f)|
(\bar\rho x+1)^\delta$.

\subsubsection{Root Bounds for a Polynomial}

We shall now investigate how to find a good root bound~$\rho$ for~$f$.
Define the perfect root bound $\RB(f) = \max \{ |\alpha| : f(\alpha)=0 \}$;
so clearly we must have $\rho \ge \RB(f)$.  Luckily, finding suitable
values for~$\rho$ is not too hard.  We can take $\rho = C(f)$, the unique
positive root of $\tilde{f}(x) = |a_d| x^d - \sum_{i=0}^{d-1} |a_i|
x^i$ because if we have $f(\alpha)=0$ then $|a_d| |\alpha|^d = |a_d \alpha^d| = |
\sum_{i=0}^{d-1} a_i \alpha^i | \le \sum_{i=0}^{d-1} |a_i| |\alpha|^i$,
hence $\tilde{f}(|\alpha|) \le 0$.  It is reported in~\cite{DM90} that
Cauchy knew of this bound.  In general, $C(f) > \RB(f)$; indeed
in~\cite{DM90} it is shown that $\RB(f) \le C(f) \le \RB(f)\,(2^{1/d}-1)^{-1}
\approx \RB(f) \, d/\log 2$ with both limits being attainable.

Alternatively, if we prefer not to compute $C(f)$, we can obtain a
slightly looser bound using a formula given in~\cite{Zas69}:
$$
Z(f) \quad = \quad
{ 1 \over 2^{1/d}-1 } \max \left\{ \left( {|a_{d-i}| \over |a_d|} {d \choose i}^{-1} \right)^{1/i} \right\}
$$
Another formula was given as exercise~4.6.2--20 in~\cite{Knu69} (curiously,
it does not appear in the second edition~\cite{Knu81}):
$$
\quad K(f) \quad = \quad
2 \max \left\{ \left( {|a_{d-i}| \over |a_d|} \right)^{1/i} : i = 1, 2, \ldots, d \right\}
$$
In fact, these formulas merely give upper bounds for $C(f)$; nevertheless,
as Knuth showed, we know that $K(f)$ cannot exceed $2d \RB(f)$.  In practice,
we can start from the smaller of $Z(f)$ and $K(f)$ then apply a few Newton
iterations to obtain quickly a tighter upper bound for $C(f)$.  The topic of
root bounds has been much studied; some more information can be found in, for
example,~\cite{Wil61} or section~6.2 of~\cite{Yap2000}.

Now, $C(f)$ can be rather larger than $\RB(f)$, the best possible value for
$\rho$.  A simple way to compute better approximations to $\RB(f)$ is given
in~\cite{DM90}.  They use Gr\"affe's transformation repeatedly to obtain
the polynomial $f_s$ whose roots are the $2^s$-th powers of the roots of
$f$, then they use the smaller of $K(f_s)$ and $Z(f_s)$ to bound $C(f_s)$
and finally take the $2^s$-th root of the result.  Choosing $s \approx
3+\log \log d$ is enough to obtain an estimate for $\RB(f)$ within a small
constant factor of optimal.  As we shall be computing powers of~$\rho$ in
the binomial expansion a larger value of~$s$ may be better for us,
in~\cite{DM90} they suggest using the value $s=\max\{3, \log d\}$.

\medskip
\noindent
{\small
{\bf Note}
Although applying Gr\"affe's transformation many times will surely produce a
better estimate, a single application of Gr\"affe's
transformation may lead to larger estimates for the largest root:
\eg~let $f=x^2-2$, then we immediately see that $3/2$ is an upper bound for
$C(f)$; however, $g=x^2-4x+4$ is the Gr\"affe transform of~$f$ and $C(g) =
2+2\sqrt{2} \approx 4.8$ which is considerably larger than $(3/2)^2 = 2.25$.
}

\subsubsection{A refinement of the binomial bound}

It is possible to refine the binomial bound if we have more detailed
information about the roots of~$f$.  Specifically, if we know values
$\rho_1 \ge \rho_2 \ge \cdots \ge \rho_d$ where each satisfies $\rho_i \ge
|\alpha_i|$ for a suitable numbering of the complex roots of~$f$, then the
magnitude of the coefficient of $x^i$ in a factor of degree $\delta$ is
bounded by the coefficient of $x^i$ in the product $\prod_{j=1}^\delta
(x+\rho_j)$.  In the rest of this paper we shall use just the standard form
of the binomial bound.

% \subsubsection{Example: behaviour of the root bounds}

% Here are some examples showing the various possible types of behaviour of
% these root bounds: there are cases when each of $C(f)$, $Z(f)$ and $K(f)$
% is arbitrarily close to the magnitude of the largest root, sometimes $Z(f)$
% is smaller than $K(f)$, sometimes {\it vice versa\/}.  Note that we never
% overestimate by a factor greater than twice the degree of the polynomial.
% For some entries we have approximated $(2^{1/n}-1)^{-1}$ by the value
% $n/\ln 2$ for compactness.  Also the entries $2-2^{-n}$ are only
% approximate.

% $$
% \begin{array}{ccccc}
% f                         & \max |\alpha_i| & C(f)     & Z(f)    & K(f) \\
% x^n-1                     & 1               & 1        & n/\ln 2 & 2    \\
% (x+1)^n                   & 1               & n/\ln 2  & n/\ln 2 & 2n   \\
% x^n - \sum_{i=0}^{n-1} x^i & 2-2^{-n}        & 2-2^{-n} & n/\ln 2 & 2    \\
% 2x^n - (x+1)^n            & n/\ln 2         & n/\ln 2  & n/\ln 2 & 2n   \\
% \end{array}
% $$

\subsection{Mignotte's Bound}

Mignotte~\cite{Mig74} has refined an inequality he ascribed to Mahler to obtain
the following bound on a factor $g$ of degree $\delta$:
$$
|g|_1 \quad \le \quad  2^\delta M(g)
      \quad \le \quad  2^\delta M(f)
      \quad \le \quad  2^\delta |f|_2
$$

In fact, formula (2) in that paper gives an individual bound for each 
coefficient of $g$: namely $|b_i| \le {\delta \choose i} M(f)$.
Mignotte has published some other bounds but the one given here seems most
relevant to our purposes.  The bound is clearly invariant under reversal
(since Mahler's measure is invariant).

\subsection{The Knuth--Cohen Bound}\label{knuth-cohen-bound}

A slight refinement of Mignotte's bound appears as exercise~4.6.2--20
in~\cite{Knu81}; it is also given as Theorem~3.5.1 in~\cite{Coh95}.  This
refined bound is sufficiently different from Mignotte's to merit separate
mention.  The bound is:
$$
|b_i| \le {\delta-1 \choose i} |f|_2 + {\delta-1 \choose i-1} |\lc(f)|
$$
and presumably for $i=0$ we simply take $|b_0| \le |a_0|$.

Like the binomial bound, this bound is not invariant under reversal so we
can improve it by using the reversal trick.  We shall use this improved
version for the comparisons below.

\subsection{Beauzamy's Bound}

In 1992 Beauzamy derived a new bound~\cite{Bea92} from a result of
Bombieri.  Once again he gives individual bounds for each $|b_i|$:
$$
|b_i| \quad \le \quad
\sqrt{{1 \over 2}{\delta \choose i}{d \choose \delta}} \; [f]_2
$$
where $[f]_2 = \sqrt{\sum_{j=0}^d |a_j|^2/{d \choose j}}$ is a weighted
norm derived from Bombieri's norm.  An interesting feature of this norm is
that the central coefficients of~$f$ have particularly small weights, so
polynomials whose central coefficients are larger than the peripheral ones
(as often happens in practice) have a small norm.  The bound is clearly
invariant under reversal (since Bombieri's norm is invariant).

\subsection{Comparison of the Degree Aware Bounds}

Here we give four example polynomials to show that each of the degree aware
bounds can be the best one; then we give a fifth example where combining
the bounds is best.  We produced the examples by generating many random
irreducible polynomials of degree~$4$ then we chose for each bound a
product of two of these polynomials where that bound was better than the
others.  Since we are interested in factors in $\ZZ[x]$ we have rounded
down to integers the values of the bounds in the tables.  It is quite
apparent that the bounds are often very loose.

%see Work/CoCoA/CoCoA4/CoeffBounds.coc

\subsubsection{A case favourable for the binomial bound}

Let us take
%%1.303703704*10^(-1)  for poly (x^4 + 4x^3 + 16x^2 + 9x - 1)*(x^4 + 4x^3 + 15x^2 + 3x - 2)
\begin{eqnarray*}
f &=& (x^4 + 4x^3 + 16x^2 + 9x - 1) (x^4 + 4x^3 + 15x^2 + 3x - 2)\cr
  &=& x^8 + 8x^7 + 47x^6 + 136x^5 + 285x^4 + 171x^3 - 20x^2 - 21x + 2
\end{eqnarray*}
Factorizing~$f$ modulo~$5$ shows that the only possible degree for a true factor is~$4$.
For this polynomial we obtained a good root bound $\rho \approx
3.84$.  Our estimate for Mahler's measure is about~$197$.
And Bombieri's norm is about~$47$.  The coefficient bounds are:
\begin{center}
\begin{tabular}{|l|rrrrr|r|}
\hline
            & $x^4$ & $x^3$ & $x^2$ & $x^1$ & $x^0$ & Overall \\
\hline
Binomial    &    1  &   15  &   88 &   84  &    2  &      88 \\
Mignotte    &  196  &  787  & 1181 &  787  &  196  &    1181 \\
Beauzamy    &  275  &  551  &  675 &  551  &  275  &     675 \\
Knuth--Cohen&    1  &  366  & 1093 &  369  &    2  &    1093 \\
\hline
\end{tabular}
\end{center}
In the table above, and those of the subsections below, we see the
values of the individual coefficient bounds for each method, and also
the resulting overall height bound.  In this instance we observe that
the binomial method produces a considerably smaller overall value
than the other methods, and so is the best in this case.

\subsubsection{A case favourable for Mignotte's bound}

Let us take
%%4.555808656*10^(-1)  for poly (x^4 - 7x^3 + 7x^2 - 8x + 2)*(2x^4 - 2x^3 - 2x^2 + 6x - 5)
\begin{eqnarray*}
f &=& (x^4 - 7x^3 + 7x^2 - 8x + 2) (2x^4 - 2x^3 - 2x^2 + 6x - 5)\cr
  &=& 2x^8 - 16x^7 + 26x^6 - 10x^5 - 41x^4 + 89x^3 - 87x^2 + 52x - 10
\end{eqnarray*}
Factorizing~$f$ modulo~$37$ shows that the only possible degree for a true
factor is~$4$.  We estimate the Mahler measure of~$f$ to be about~$33.4$.  A
good root bound for $f$ is $\rho \approx 6.1$, and for its reverse we get
$\bar\rho \approx 3.2$.  The Bombieri norm of~$f$ is about~$31$.  The
resulting coefficient bounds are:
\begin{center}
\begin{tabular}{|l|rrrrr|r|}
\hline
            & $x^4$ & $x^3$ & $x^2$ & $x^1$ & $x^0$ & Overall \\
\hline
Binomial    &    2  &   48  &  439 &  129  &  10   &     439 \\
Mignotte    &   33  &  133  &  200 &  133  &  33   &     200 \\
Beauzamy    &  180  &  361  &  443 &  361  & 180   &     443 \\
Knuth--Cohen&    2  &  150  &  440 &  174  &  10   &     440 \\
\hline
\end{tabular}
\end{center}

% We note that, while not the best overall, the binomial bound gives the lowest
% limits for all coefficients except that of $x^2$.

\subsubsection{A case favourable for the Knuth--Cohen bound}

Let us take
%%5.548098434*10^(-1)  for poly (x^4 - 8x^3 - 7x^2 - 5x + 5)*(2x^4 - 6x^3 - x^2 + 4x - 5)
\begin{eqnarray*}
f &=& (x^4 - 8x^3 - 7x^2 - 5x + 5) (2x^4 - 6x^3 - x^2 + 4x - 5)\cr
  &=& 2x^8 - 22x^7 + 33x^6 + 44x^5 + 10x^4 - 13x^3 + 10x^2 + 45x - 25
\end{eqnarray*}
Factorizing~$f$ modulo~$11$ shows that the only possible degree for a true
factor is~$4$.  We estimate the Mahler measure of~$f$ to be about $75$.  A
good root bound for~$f$ is $\rho \approx 8.9$, and for its reverse we get
$\bar\rho \approx 2.05$.  The Bombieri norm of~$f$ is about $32$.  The
resulting coefficient bounds are:
\begin{center}
\begin{tabular}{|l|rrrrr|r|}
\hline
            & $x^4$ & $x^3$ & $x^2$ & $x^1$ & $x^0$ & Overall \\
\hline
Binomial    &    2 &    70  &  628 &  204  &  25   &     628 \\
Mignotte    &   74 &   298  &  447 &  298  &  74   &     447 \\
Beauzamy    &  189 &   378  &  463 &  378  & 189   &     463 \\
Knuth--Cohen&    2 &    86  &  248 &  155  &  25   &     248 \\
\hline
\end{tabular}
\end{center}

\subsubsection{A case favourable for Beauzamy's bound}

Let us take
%%6.078184111*10^(-1)  for poly (x^4 + 5x^3 - 14x^2 - x - 3)*(3x^4 + 8x^3 + 15x^2 - 5x - 1)
\begin{eqnarray*}
f &=& (x^4 + 5x^3 - 14x^2 - x - 3) (3x^4 + 8x^3 + 15x^2 - 5x - 1)\cr
  &=& 3x^8 + 23x^7 + 13x^6 - 45x^5 - 253x^4 + 26x^3 - 26x^2 + 16x + 3
\end{eqnarray*}
Factorizing $f$ modulo~$37$ shows that the only possible degree for a true
factor is~$4$.  A good root bound for~$f$ is $\rho \approx 7.0$; for the
reverse we get $\bar\rho \approx 7.0$.  The Mahler measure of~$f$ is about
$259$.  The Bombieri norm of~$f$ is about $33$.  The resulting coefficient
bounds are:
\begin{center}
\begin{tabular}{|l|rrrrr|r|}
\hline
            & $x^4$ & $x^3$ & $x^2$ & $x^1$ & $x^0$ & Overall \\
\hline
Binomial    &    3  &   83  &  880 &   83  &   3   &     880 \\
Mignotte    &  258  & 1035  & 1553 & 1035  & 258   &    1553 \\
Beauzamy    &  197  &  394  &  482 &  394  & 197   &     482 \\
Knuth--Cohen&    3  &  270  &  793 &  270  &   3   &     793 \\
\hline
\end{tabular}
\end{center}

% Again the binomial bound is (equal) best except for the coefficient of $x^2$.

\subsubsection{A case where the combined bound is best}
\label{combined-is-best}

Let us take
%%5.337078652*10^(-1)  for poly (2x^4 + 8x^3 + 10x^2 + 9x - 1)*(x^4 - 3x^3 + 5x^2 - 5)
\begin{eqnarray*}
f &=& (2x^4 + 8x^3 + 10x^2 + 9x - 1) (x^4 - 3x^3 + 5x^2 - 5)\cr
  &=& 2x^8 + 2x^7 - 4x^6 + 19x^5 + 12x^4 + 8x^3 - 55x^2 - 45x + 5
\end{eqnarray*}
Factorizing~$f$ modulo~$61$ shows that the only possible degree for a true
factor is~$4$.  We obtain the following coefficient bounds:
\begin{center}
\begin{tabular}{|l|rrrrr|r|}
\hline
            & $x^4$ & $x^3$ & $x^2$ & $x^1$ & $x^0$ & Overall \\
\hline
Binomial    &    2  &   22 &   95  &  178  &   5   &     178 \\
Mignotte    &   63  &  255 &  382  &  255  &  63   &     382 \\
Beauzamy    &  118  &  236 &  290  &  236  & 118   &     290 \\
Knuth--Cohen&    2  &   81 &  231  &   90  &   5   &     231 \\
\hline
\end{tabular}
\end{center}

If we consider only the overall height bounds produced by each method then
the best we can conclude is that any factor of degree~$4$ can have height
at most $178$.  However, taking together the bounds for each individual
coefficient (\ie~the minimum of each column) we see that no coefficient can
exceed $95$.

\subsection{Preprocessing to get Better Bounds}

From the examples above it is quite clear that often all the bounds are
very loose.  Now, any of the bounds above applied to some multiple of $f$
will naturally be valid also for factors of~$f$ itself.  So we are
interested in finding small height multiples of $f$.  We find that
Theorems~$4$ and~$4'$ in~\cite{Mig88} prove the existence of small height
multiples of polynomials in~$\ZZ[x]$; unfortunately, they do not give a
good way of actually finding them.  One approach is to use LLL lattice
reduction~\cite{LLL82} to look for a multiple of~$f$ with small $l_2$ norm.

Another approach is hinted at in Theorem~B of~\cite{CMP87}: given a
positive degree~$\hat d$ find the unique monic $\hat f_{\hat d} \in \QQ[x]$
of degree~$\hat d$ which minimizes $|f \hat f_{\hat d}|_2$.  The
coefficients of $\hat f_{\hat d}$ are the solutions of a simple linear
system, so we have an efficient and effective method for finding a small
multiple of prescribed degree in $\QQ[x]$.

%  We just multiply~$f$ by a monic polynomial of degree~$\hat d$
% having indeterminate coefficients.  The square of the $l_2$ norm is then
% just a quadratic polynomial in these indeterminates, and we find its
% minimum simply by setting all first derivatives to zero.  This produces a
% linear system in the indeterminates whose solution provides us with the
% coefficients of the polynomial~$\hat f_{\hat d}$.

\medskip

The principal drawback of this preprocessing idea is that only rarely does
it produce better bounds: the reduction in height is not usually sufficient
to offset the increase in degree~---~recall that all the degree aware
bounds grow exponentially with degree.

% In some cases we can transform the factorization into an equivalent smaller task.
% For instance, we can look for a linear transformation $x \mapsto ax+b \in \ZZ[x]$ 
% such that $f(x) = g(ax+b)$ where $g$ has smaller coefficients.  A heuristic
% for this is implemented in CoCoA as the function {\tt LinearSimplify}.

%%%%%%%%%%%%%%%%%%%%%%%%%%%%%%%%%%%%%%%%%%%%%%%%%%%%%%%%%%%%%%%%%%%%%%%%%%%%%
\section{Single Factor Bounds}
\label{single-factor-bounds}

We now turn our attention to the single factor bounds which we mentioned at the start of
section~\ref{deg-aware-bounds}.  Even though our primary interest is in
factors in $\ZZ[x]$, once again we need to consider what happens in
$\CC[x]$.  Let $f \in \CC[x]$ have degree~$d$.  Then a {\bf single factor
  bound} for $f$ is a value $B$ such that for any non-trivial factorization $f = g_1
g_2 \cdots g_s \in \CC[x]$ at least one factor (wlog $g_1$) satisfies
$\height(g_1) \le B$.  As already observed, any of the degree aware bounds
from the previous section can be used as a single factor bound simply by
computing the bounds for a factor of degree $\lfloor d/2 \rfloor$.  Here we
see two other ways of obtaining such bounds.  We start with a new bound
which was latent in an article of Mignotte.

\subsection{Mignotte's Bound}

It is immediate from Theorem~2 in~\cite{Mig74} that in any non-trivial
factorization $f$ must have at least one factor, $g$, satisfying
$$
|g|_1 \quad \le \quad \sqrt{2^d M(f)} 
      \quad \le \quad \sqrt{2^d |f|_2}
$$
though this fact is not exploited there.  For some reason this was overlooked
also in~\cite{BTW93}: they just took the simpler bound from Mignotte's Theorem~2
and applied it to degree $\lfloor d/2 \rfloor$.

In fact, tracing through Mignotte's reasoning one can also obtain another,
more convenient version which bounds directly the sizes of the coefficients
of some factor.  This alternative version is also slightly tighter.  We now
present this new bound.

Suppose that $f = g_1 g_2$ is a non-trivial factorization with $d_1 = \deg(g_1)$ and
$d_2 = \deg(g_2)$.  Clearly $M(f) = M(g_1) M(g_2)$.  Now $|g_1|_\infty \le
{d_1 \choose \lfloor d_1/2 \rfloor} M(g_1)$, and similarly for $g_2$.
Since $d=d_1+d_2$ and both $d_1$ and $d_2$ are strictly positive,
we have ${d_1 \choose \lfloor d_1/2 \rfloor}{d_2 \choose
  \lfloor d_2/2 \rfloor} \le \frac{2}{3}{d \choose \lfloor d/2 \rfloor}$.
Putting it all together we find that
$$
|g_1|_\infty |g_2|_\infty \le \frac{2}{3} {d \choose \lfloor d/2 \rfloor} M(f)
$$
We may assume that $|g_1|_\infty \le |g_2|_\infty$; therefore $|g_1|_\infty
\le \sqrt{\frac{2}{3} {d \choose \lfloor d/2 \rfloor} M(f)}$.  Using
Stirling's approximation we can estimate the binomial coefficient to be
about $2^{d+1}/(2 \pi d)^{1/2}$.  If we want, using Theorem~1
from~\cite{Mig74}, we can replace $M(f)$ by the bound $|f|_2$ to obtain a
closed form.  Computationally it is better to use approximations
from~\cite{CMP87} or~\cite{DM90} to estimate $M(f)$.

\subsection{The BTW Bound}

Another single factor bound was described more recently in~\cite{BTW93}; it is based
on Bombieri's weighted norm.  They show that at least one factor, $g$, must
satisfy
$$
|g|_\infty \le c \sqrt{ 2^d d^{-3/4} \, [f]_2}
$$
for some constant $c < 1.1$ (provided $d = \deg(f) > 2$).  It is
fairly evident that for large~$d$ Mignotte's single factor bound is
greater than BTW if we choose to use $|f|_2$ in place of $M(f)$ since
$|f|_2 \ge [f]_2$.  However, if we compute a good approximation to $M(f)$
then Mignotte's bound can be substantially lower: \eg~consider $f=(x+1)^{2d}$
for which $M(f)=1$ and $[f]_2 = 2^d$, then Mignotte's bound yields $2^d
\bigl( \pi d \bigr)^{-1/4}$ whereas the BTW bound is $c\, 2^{3d/2}
(2d)^{-3/8}$~---~almost 50\% bigger than Mignotte's bound, in logarithmic
terms.

\subsection{Comparison of the Single Factor Bounds}

Here are two examples to show that neither bound is superior to the other.
Obviously the value of Mignotte's bound depends on the quality of the
approximation to $M(f)$ we use: for our computations we used the techniques
in~\cite{CMP87} to obtain a fairly good approximation.  The first example is
%%4.468085106*10^(-1)  for poly (2x^4 + 5x^3 + 7x^2 + 6x + 3)*(3x^4 + 6x^3 + 7x^2 + 5x + 2)
%%3.934426230*10^(-1)  for poly (3x^4 + 7x^3 + 9x^2 + 7x + 3)*(4x^4 + 7x^3 + 9x^2 + 7x + 4)
\begin{eqnarray*}
  f&=& 6x^8 + 27x^7 + 65x^6 + 105x^5 + 123x^4 + 105x^3 + 65x^2 + 27x + 6\cr
   &=& (2x^4 + 5x^3 + 7x^2 + 6x + 3) (3x^4 + 6x^3 + 7x^2 + 5x + 2)
\end{eqnarray*}
where we find that Mignotte's bound is $21$ whereas the BTW bound is $47$.  Conversely, if we look at the example
%%4.400000000*10^(-1)  for poly (x^4 + 2x^3 + 2x^2 - 2x + 1)*(x^4 - 2x^3 + 2x^2 + 2x + 1)
%%4.038461538*10^(-1)  for poly (x^4 - 4x^3 + 5x^2 + 4x + 1)*(x^4 + 4x^3 + 5x^2 - 4x + 1)
%%4.017857143*10^(-1)  for poly (x^4 - 6x^3 + 14x^2 + 6x + 1)*(x^4 + 6x^3 + 14x^2 - 6x + 1)
\begin{eqnarray*}
  f&=& x^8 - 6x^6 + 59x^4 - 6x^2 + 1\cr
   &=& (x^4 - 4x^3 + 5x^2 + 4x + 1) (x^4 + 4x^3 + 5x^2 - 4x + 1)
\end{eqnarray*}
we find that the BTW bound is $21$ whereas Mignotte's is $52$.

Once again, since neither bound is always better than the other, it is
worth computing both, and taking the smaller result.  In the next
section we compare the single factor bounds with the degree aware ones.

\section{Single Factor {\it vs.} Degree Aware Bounds}\label{single-vs-deg-aware}

In this section we consider the question {\it Which are better: single
  factor bounds or degree aware bounds?}  The answer is {\it Neither.}  We
shall see that the best answer is always to calculate both and pick
whichever is the smaller.  For simplicity we shall restrict attention to
short factorizations: $f = g_1 g_2$.

One situation apparently favourable for the degree aware bounds is when we
know that one true factor must have particularly low degree
(\ie~considerably smaller than $d/2$)~---~recall that we might be able to
learn this from degree analysis.  Conversely, a situation apparently
favourable for the single factor bounds is when both factors could have
high degree (\ie~close to $d/2$).  We give two examples of both types: one
where the single factor bounds are better and one where the degree aware
bounds are better.

\subsection{Factorizations with a Low Degree Factor}

We present two example short factorizations both with factors of
degree~$2$ and~$18$.  In one case we see that the degree aware bounds
are much smaller than the single factor bounds, in the other case we
see the opposite.  For convenience, these two examples have a special
structure: the smaller factor $g_1$ is quadratic, the larger factor
$g_2 = g_1^9+x^8$, which can be viewed as a ``small perturbation'' of
$g_1^9$.

Taking $g_1 = x^2+5x+9$ we obtain a polynomial $f$ whose factorization
modulo $151$ tells us that the only degree pattern for true factors is
$2+18$.  A good root bound for $f$ is about $3.9$, and the binomial bound
for a degree~$2$ factor tells us that the maximum coefficient cannot
exceed~$15$; the other degree aware bounds are larger.  In contrast,
Mignotte's single factor bound is larger than $2 \times 10^7$; the BTW
bound is larger still.

Taking instead $g_1 = 8x^2+9x+9$ we obtain a polynomial $f$ whose
factorization modulo $61$ tells us that the only degree pattern for true
factors is $2+18$.  A good root bound for $f$ is about $1.3$, and the
binomial bound for a degree~$2$ factor tells us that the maximum
coefficient cannot exceed about~$2.8 \times 10^9$; the other degree aware
bounds are even larger.  In contrast, Mignotte's single factor bound is
less than $2.7 \times 10^7$; the BTW bound is larger.

% In this first example we show how knowledge that a true factor must be of
% low degree lets us obtain better values from a degree aware bound than from
% a single factor bound.  We constructed the example using a method that
% often seems to work well.  First we chose the low degree factor, $g_1$, in
% this case an irreducble quadratic.  Then we set the high degree factor $g_2
% = g_1^k + \varepsilon$ where $\varepsilon$ is a small polynomial chosen to make
% $g_2$ irreducible~---~in this case we took $k=3$ and $\varepsilon=x^2$.  Not
% all choices of $g_1$ produce a suitable example, but many do.
% %% This is an example found using a random search:
% %%1.363636364*10^(-1)  for poly (x^2 + 2x + 5) (x^6 - 3x^5 + 3x^4 + 7x^3 - 6x^2 + 7x + 5)
% \begin{eqnarray*}
%   f&=&x^8 + 20x^6 + 151x^4 + 505x^2 + 625\cr
%    &=&(x^2 + 5) (x^6 + 15x^4 + 76x^2 + 125)
% \end{eqnarray*}
% Factorizing $f$ modulo~$11$ shows that the only possible degrees for true
% factors are~$2$ and~$6$~---~we are interested in bounding the lower degree factor.
% We find $\rho \approx 2.62$ is a good root bound for $f$, and the
% binomial bound tells us that any factor of degree~$2$ has height at
% most~$6$.  In contrast the single factor bounds are rather larger:
% Mignotte's bound gives~$170$ while the BTW bound produces~$202$.

\subsubsection{A Family having Large Single Factor Bounds}

We can, with very high probability, construct examples for which the degree
aware binomial bound is arbitrarily many times smaller than either single
factor bound.  The idea is simple: we construct a polynomial which has two
irreducible factors, one of very low degree and one of high degree.  The
degree aware bounds have the advantage of being able to use the information
that there is a very low degree factor.

Consider the polynomial $f = (x+1) (x^{d-1} - 10^{d-1} +
\varepsilon)$ where $\varepsilon$ is a polynomial of degree less than
$d-1$ having small coefficients and such that~$f$ has only two irreducible
factors over $\ZZ$.  It seems quite easy to find such~$\varepsilon$ for which
degree analysis proves that the only possible degrees for true factors
are~$1$ and~$d-1$.  By choosing $d$ sufficiently large (and $\varepsilon$
sufficiently small), Knuth's root bound for~$f$ yields a value of
about~$20$, so a good root bound will surely be no larger than this.  Thus
the binomial bound for the height of a factor of degree~$1$ is at most~$20$
whereas the BTW single factor bound is greater than $20^{(d-1)/2}
d^{-3/8}$, and Mignotte's single factor bound is greater than $20^{(d-1)/2}
\bigl(2 \pi d \bigr)^{-1/4}$ (since $M(f) \ge 10^{d-1}$).  In other words
the binomial bound is about $(d-1)/2$ times smaller (in a logarithmic
sense) than either of the single factor bounds.

\subsection{Factorizations with High Degree Factors}

% Some examples in lower degrees
% Best so far in deg 10 is 0.407 STILL RUNNING on point!!!
% 4.072081764*10^(-1)  for poly (x^10 + x^9 + x^8 + x^7 + x^6 + x^5 + x^4 + 2x^3 + 4x^2 + 4x + 5)*(x^10 + x^9 + 2x^8 + 2x^7 + 3x^6 + 4x^5 + 4x^4 + 4x^3 + 5x^2 + 5x + 5)
% 8.185840708*10^(-1)  for poly (x^9 - x^6 + 1)*(x^9 - x^3 + 1)
% 3.974895397*10^(-1)  for poly (x^8 + x^7 + x^6 + x^5 + x^4 + x^3 + x^2 + 2x + 5)*(x^8 + x^7 + x^6 + x^5 + x^4 + 2x^3 + 3x^2 + 3x + 5)
% 8.030303030*10^(-1)  for poly (x^7 + x^6 + x^3 - x - 1)*(x^7 + x^6 - x^4 - x - 1)
% 4.193548387*10^(-1)  for poly (x^6 + x^5 + x^4 + x^3 + x^2 + 2x + 5)*(x^6 + x^5 + x^4 + x^3 + 2x^2 + 3x + 5)
% 4.687500000*10^(-1)  for poly (x^5 + x^4 + x^3 + x^2 + 3x + 5)*(x^5 + x^4 + 2x^3 + 3x^2 + 3x + 5)
% 4.411764706*10^(-1)  for poly (x^4 + x^3 + x^2 + 2x + 5)*(x^4 + x^3 + 2x^2 + 2x + 5)
% Below is a much better example, but it has larger factors...
% 3.023255814*10^(-1)  for poly (x^4 + x^3 + x^2 + 4x + 19)*(x^4 + x^3 + 4x^2 + 5x + 19)

If degree analysis does not exclude the existence of high degree true
factors (\ie~of degree close to $d/2$) then the single factor bounds are
often better than the degree aware bounds.  Nonetheless, it is not hard to
find cases where the degree aware bounds are tighter than the single factor
bounds even in these ``unfavourable circumstances''.  For instance, let us
take the irreducible factors
\begin{eqnarray*}
g_1 &=& x^{10} + x^9 + x^8 + x^7 + x^6 + x^5 + x^4 + 2x^3 + 4x^2 + 4x + 5 \cr
g_2 &=& x^{10} + x^9 + 2x^8 + 2x^7 + 3x^6 + 4x^5 + 4x^4 + 4x^3 + 5x^2 + 5x + 5
\end{eqnarray*}
Then for the product $f=g_1 g_2$ we see from the factorization modulo~$463$
that the only possible degree for a true factor is~$10$.  A good root bound
for $f$ is $1.24$, and the binomial bound for a degree~$10$ factor is thus
$757$.  The single factor bounds are rather larger: Mignotte's bound gives
$1859$, and the BTW bound is a little larger at $1920$.

\medskip
We now present a case where the single factor bounds are significantly smaller
than the degree aware bounds.  We take as irreducible factors
% 4.363792154*10^(-2)  for poly (x^10 + 5x^9 - 5x^8 - 3x^7 + 5x^6 - 5x^5 - 2x^4 - 4x^3 - 5x^2 + x + 3)*(-x^10 - x^9 + 3x^7 + x^6 - 5x^5 - x^4 + 4x^3 - 5x^2 + 1)
\begin{eqnarray*}
g_1 &=& x^{10} + 5x^9 - 5x^8 - 3x^7 + 5x^6 - 5x^5 - 2x^4 - 4x^3 - 5x^2 + x + 3 \cr
g_2 &=& x^{10} + x^9 - 3x^7 - x^6 + 5x^5 + x^4 - 4x^3 + 5x^2 - 1
\end{eqnarray*}
Then for the product $f=g_1 g_2$ we see from the factorization modulo~$587$
that the only possible degree for a true factor is~$10$.  The best degree
aware bound is Knuth--Cohen which gives~$16339$.  The single factor bounds
are rather smaller: Mignotte's bound gives $3071$, and the BTW bound is
smaller still at $713$.

%%%%%%%%%%%%%%%%%%%%%%%%%%%%%%%%%%%%%%%%%%%%%%%%%%%%%%%%%%%%%%%%%%%%%%%%%%%%%
%%%%%%%%%%%%%%%%%%%%%%%%%%%%%%%%%%%%%%%%%%%%%%%%%%%%%%%%%%%%%%%%%%%%%%%%%%%%%
\section{Polynomials in $\ZZ[x]$ having a large factor}
\label{one-large-factor}

Looking at the examples in
sections~\ref{deg-aware-bounds},~\ref{single-factor-bounds}
and~\ref{single-vs-deg-aware}, it is quite clear that the factor bounds are
often very loose.  Indeed, based on experience, it is natural to conjecture
that the factors must have height no greater than that of the polynomial
they divide.  Remarkably, this is false: there are cyclotomic polynomials
of arbitrarily great height~\cite{Vau74}, yet they each divide a polynomial
of the form $x^d-1$ (for some~$d$ depending on the polynomial).

In this section we are interested in polynomials which have at least one
large height factor.  The existence of these polynomials having a particularly
large factor helps explain why the degree aware bounds from
section~\ref{deg-aware-bounds} have to be so generous: those bounds must be
large enough to accommodate these unusually large factors.  In contrast,
these factorizations do not have similar implications for the single factor
bounds of section~\ref{single-factor-bounds}~---~in
section~\ref{high-ratio} we will consider large ratio factorizations which
do force also the single factor bounds to be large.

\medskip

Our main interest is in irreducible factors in $\ZZ[x]$; and in the context
of polynomial factorization we can further restrict to irreducible factors
of primitive square-free polynomials.  However, the known bounds apply to
any factor of any polynomial in $\CC[x]$ under the scaling hypothesis
mentioned at the start of section~\ref{deg-aware-bounds}.  As we shall
see, this wider applicability forces the bounds to be much larger than
necessary for irreducible factors in $\ZZ[x]$.

It appears to be very difficult to devise general bounds valid only for
factors in $\ZZ[x]$~---~or ideally, irreducible factors in $\ZZ[x]$.  It is
even unclear how to devise a bound valid only for factors of a primitive
square-free polynomial in $\ZZ[x]$.  Yet the examples below indicate that
such specialized bounds could be significantly tighter than the current
ones we know.

\subsection{Large Factors of $x^d-1$}\label{large-factors-xd-1}

In this subsection we look at how large factors of $x^d-1$ can be.
We chose to consider this family of polynomials for several reasons:
the family depends on a single parameter (namely the degree~$d$),
there are already many interesting theoretical results, and the
limited nature of the family permits computational experimentation
up to moderately high degree.

\subsubsection{Large height irreducible factors of $x^d-1 \in \ZZ[x]$}

The irreducible factors of $x^d-1$ in $\ZZ[x]$ are called {\bf
  cyclotomic polynomials}.  They enjoy numerous special properties; we
recall just one of them: $x^d-1 = \prod_{n|d} \phi_n(x)$ where we
write $\phi_n(x)$ to denote the $n$-th cyclotomic polynomial.  The
first few cyclotomic polynomials all have height~$1$, but as we reach
higher indices we find examples of greater height.  Here is a table of
heights of certain cyclotomic polynomials: these are successive
maximums (up to index 100000).  We used version~4.7.4 of the CoCoA
system~\cite{CoCoA} for the computations.

\begin{center}
\label{tab:cycloheightirred}
\begin{tabular}{|r|r|l|}
\hline
$\height(\phi_d)$  &  $d$ & Factorization of $d$ \\
\hline
     2     &       105 & $3 \cdot 5 \cdot 7                   $ \\
     3     &       385 & $5 \cdot 7 \cdot 11                  $ \\
     4     &      1365 & $3 \cdot 5 \cdot 7 \cdot 13          $ \\
     5     &      1785 & $3 \cdot 5 \cdot 7 \cdot 17          $ \\
     6     &      2805 & $3 \cdot 5 \cdot 11 \cdot 17         $ \\
     7     &      3135 & $3 \cdot 5 \cdot 11 \cdot 19         $ \\
     9     &      6545 & $5 \cdot 7 \cdot 11 \cdot 17         $ \\
    14     &     10465 & $5 \cdot 7 \cdot 13 \cdot 23         $ \\
    23     &     11305 & $5 \cdot 7 \cdot 17 \cdot 19         $ \\
    25     &     17255 & $5 \cdot 7 \cdot 17 \cdot 29         $ \\
    27     &     20615 & $5 \cdot 7 \cdot 19 \cdot 31         $ \\
    59     &     26565 & $3 \cdot 5 \cdot 7 \cdot 11 \cdot 23 $ \\
   359     &     40755 & $3 \cdot 5 \cdot 11 \cdot 13 \cdot 19$ \\
\hline
\end{tabular}
\end{center}

We notice that one needs to go to very high degrees to obtain large
coefficients.  These cyclotomic polynomials have much smaller heights than
allowed by the factor bounds: \eg~even for the smallest case in the table,
the best bound we get for a degree~$48$ factor of $x^{105}-1$ is ${48
  \choose 24} \approx 3\times 10^{13}$ which is far larger than the actual
height.  In~\cite{Vau74} Vaughan showed that the asymptotic growth of
$\height(\phi_d)$ for certain~$d$ increases as $\exp(d^{(\log
  2+o(1))/\log\log d})$.

% Here are some further spot values; note that these are not successive maxima.
% \begin{tabular}{ll}
%  Height  &     Index \\
%    532   &     $255255 = 3 \cdot 5 \cdot 7 \cdot 11 \cdot 13 \cdot 17$ \\
%   1182   &     $285285 = 3 \cdot 5 \cdot 7 \cdot 11 \cdot 13 \cdot 19$ \\
%   1311   &     $345345 = 3 \cdot 5 \cdot 7 \cdot 11 \cdot 13 \cdot 23$ \\
%   5477   &     $373065 = 3 \cdot 5 \cdot 7 \cdot 11 \cdot 17 \cdot 19$ \\
%   2006   &     $435435 = 3 \cdot 5 \cdot 7 \cdot 11 \cdot 13 \cdot 29$ \\
%   2210   &     $440895 = 3 \cdot 5 \cdot 7 \cdot 13 \cdot 17 \cdot 19$ \\
%    618   &     $451605 = 3 \cdot 5 \cdot 7 \cdot 11 \cdot 17 \cdot 23$ \\
%   1123   &     $465465 = 3 \cdot 5 \cdot 7 \cdot 11 \cdot 13 \cdot 31$ \\
%    783   &     $504735 = 3 \cdot 5 \cdot 7 \cdot 11 \cdot 19 \cdot 23$ \\
% 669606   &    $4849845 = 3 \cdot 5 \cdot 7 \cdot 11 \cdot 13 \cdot 17 \cdot 19$ \\
% \end{tabular}

\subsubsection{Large height reducible factors of $x^d-1$}
%% see BigCycloFac.coc

Recalling that the bounds of section~\ref{deg-aware-bounds} do not distinguish between reducible and
irreducible factors, we now look for large factors of $x^d-1$ including
reducible ones.  Here is a table of successive maximums of the greatest
height factor in $\ZZ[x]$ of $x^d-1$ for~$d$ up to~$719$.  The values were
computed directly: factorize $x^d-1$ then try all possible products of the
irreducible factors.  We did not attempt $d=720$ as $x^{720}-1$ has~$30$
factors, and it would take a long time to generate and test all $2^{30}$
factors.  These examples were computed by a dedicated C++ program which
relied upon version 0.99 of CoCoALib~\cite{CoCoA}.

\begin{center}
\label{tab:cycloheight-reducible}
\begin{tabular}{|r|r|l|}
\hline
Factor Height & $d$ & Factor \\
\hline
    $3\quad$  &    12  & $\phi_1 \phi_4 \phi_6$ \\ %% deg=5
    $4\quad$  &    20  & $\phi_2 \phi_4 \phi_5$ \\ %% deg=7
   $12\quad$  &    30  & $\phi_1 \phi_6 \phi_{10} \phi_{15}$ \\ %% deg=15
%%      26   &    60  & $\phi_1 \phi_4 \phi_6 \phi_{10} \phi_{60}$ \\ %% deg=
   $54\quad$  &    60  & $\phi_1 \phi_4 \phi_6 \phi_{10} \phi_{15} \phi_{60}$ \\ % deg= 33
%%      32   &    70  & $\phi_1 \phi_{10} \phi_{14} \phi_{35}$ \\
   $55\quad$  &    84  & $\phi_1 \phi_4 \phi_6 \phi_{14} \phi_{21} \phi_{84}$ \\ %% deg=47
   $58\quad$  &    90  & $\phi_2 \phi_3 \phi_5 \phi_{18} \phi_{30} \phi_{45}$ \\ %% deg=45
   $72\quad$  &   105  & $\phi_3 \phi_5 \phi_7 \phi_{105}$ \\ %% deg=60
  $192\quad$  &   120  & $\phi_2 \phi_3 \phi_4 \phi_5 \phi_{24} \phi_{30} \phi_{40} \phi_{60}$ \\ %%deg=57
%%     358   &   180  & $\phi_1 \phi_4 \phi_6 \phi_{10} \phi_{15} \phi_{18} \phi_{36} \phi_{45} \phi_{60}$ \\
  $475\quad$  &   180  & $\phi_{1} \phi_{4} \phi_{6} \phi_{10} \phi_{15} \phi_{18} \phi_{36} \phi_{45} \phi_{60} \phi_{90}$ \\ %% deg=99
$10188\quad$  &   210  & $\phi_1 \phi_6 \phi_{10} \phi_{14} \phi_{15} \phi_{21} \phi_{35} \phi_{210}$ \\ %%deg=105
$395796\quad$ &   420  & $\phi_1 \phi_4 \phi_6 \phi_{10} \phi_{14} \phi_{15} \phi_{21} \phi_{35} \phi_{60} \phi_{84} \phi_{140} \phi_{210} $ \\ %%deg=195
$396660\quad$ &   630  & $\phi_{2} \phi_{3} \phi_{5} \phi_{7} \phi_{9} \phi_{30} \phi_{42} \phi_{45} \phi_{70} \phi_{90} \phi_{105} \phi_{126} \phi_{315}$ \\ %% deg =339
\hline
\end{tabular}
\end{center}

Note how the factor height increases with index $d$ much more rapidly than
it did in the previous subsection where we measured the largest {\em
  irreducible} factor.  Nevertheless these highest factors still lie well
within the limits permitted by the factor bounds: for instance the best
bound for a degree~$57$ factor of $x^{120}-1$ is ${57 \choose 28} \approx
1.5 \times 10^{16}$~---~far larger than the observed value of~$192$.

\medskip
\noindent
{\small
{\bf Note}
In~\cite{PR07} Theorems~$4.1$ and~$5.3$ state that the asymptotic growth of
the greatest height of a factor of $x^d-1$ for certain exponents~$d$ (with many
factors) increases as $\exp(d^{(\log 3+o(1))/\log\log d})$.  This
impressive result is, however, relevant only for very large degrees
well outside the realm of practical polynomial factorization.
}

%%%%%%%%%%%%%%%%%%%%%%%%%%%%%%%%%%%%%%%%%%%%%%%%%%%%%%%%%%%%%%%%%%%%%%%%%%%%%%%%%%%%%%%%%%%%%%%%%%%

\subsubsection{Large factors of $x^d-1$ in $\RR[x]$ and $\CC[x]$}
\label{xd-1-large-RR-factor}

Recalling that the bounds in section~\ref{deg-aware-bounds} on the heights
do, in fact, apply to all factors in $\CC[x]$ which satisfy the scaling
hypothesis, we now consider such factors.  We shall see that much greater
heights can be achieved compared to what we found when looking at factors
in $\ZZ[x]$.  Our study is helped greatly by the particularly simple
factorization: $x^d-1 = \prod_{k=0}^{d-1} (x-\zeta^k) \in \CC[x]$ where
$\zeta$ is a primitive $d$-th root of unity.

Upon examining all possible factors in $\CC[x]$ for small values of~$d$, we
quickly found a simple characterization of a factor of greatest height.
There are three separate cases depending on the value of $d$ modulo $3$:
namely $d=3t$, $d=3t+1$ and $d=3t-1$.  In the
cases $d=3t$ and $d=3t-1$ a factor of greatest height is given by $g_d=
\prod_{k=1-t}^{t-1}(x-\zeta^k)$; moreover, by symmetry we see that $g_d \in
\RR[x]$.  In contrast, the case $n=3t+1$ is different: a best factor is
$h_d = g_d(x) \cdot (x-\zeta^t) \not \in \RR[x]$ and there is no best
factor which is real.  Nevertheless also in the case $d=3t+1$ we see that
$g_d$ continues to be a real factor of greatest height, and in fact it is
not much smaller than $h_d$: empirically we observe that $\height(g_d) /
\height(h_d) \ge 1/\sqrt{2}$ always, and that $\height(g_d)/\height(h_d)
\rightarrow 1$ as~$d$ increases.

Empirically, based on computed values up to $d=100$ (or equivalently
$t=33$), we find that $\height(g_d) \approx 0.15 \times 1.37^d$ or
equivalently $\log(\height(g_d)) \approx 0.316d - 1.9$.  For a degree
$\delta = \lfloor (2d+1)/3 \rfloor$ factor the best bounds are the binomial
bound (with $\rho=1$) or Mignotte's bound (with $M(f)=1$): they allow
height up to ${\delta \choose \lfloor \delta/2 \rfloor}$ which we can
approximate using Stirling's formula to obtain $2^\delta/\sqrt{\pi
  \delta/2} \sim 1.58^d/\sqrt{d}$ for large~$d$.  Here we see that the
bounds are too large by roughly 42\% in logarithmic terms.

% \medskip
% \noindent
% Note: in section~\ref{xd-1-large-ratio} we look at large ratio
% factorizations of $x^d-1$ in $\RR[x]$.

%%%%%%%%%%%%%%%%%%%%%%%%%%%%%%%%%%%%%%%%%%%%%%%%%%%%%%%%%%%%%%%%%%%%%%%%%%%%%
\subsection{Factors of Polynomials of Height 1 in $\ZZ[x]$}

In the preceding section we looked at the family $x^d-1$ which grows only
slowly with degree: in each degree there is precisely one polynomial to
consider.  Here we consider all polynomials of height~$1$ and given degree.
This family of polynomials increases exponentially with degree: there are
$4 \times 3^{d-2} $ cases to consider in degree $d$ after excluding some
obvious symmetries and multiples of $x$.  We restricted attention to
height~$1$ because allowing greater heights would result in even faster
exponential growth wih degree.  As a consequence our experimental results
extend only to moderately high degrees.

The next two subsections below look at factors in $\ZZ[x]$: first only the
irreducible factors, then all the factors.  We did not investigate factors
in $\CC[x]$ since the CoCoALib library does not (yet) have the ability to
compute approximate roots in $\CC$.

In the last subsection, we mention briefly the special case of height~$1$
multiples of powers of $x+1$, a topic which has already attracted the
attention of several authors.

%%%%%%%%%%%%%%%%%%%%%%%%%%%%%%%%%%%%%%%%%%%%%%%%%%%%%%%%%%%%%%%%%%%%%%%%%%%%%
\subsubsection{Large height irreducible factors}

In this subsection we are interested in polynomials of height~$1$ having a
large irreducible factor in $\ZZ[x]$.  We note immediately that the
greatest height in a given degree grows much faster than it did for cyclotomic
polynomials.

The entries in this table were found using an {\it ad hoc\/} program
written in C++ which relied upon CoCoALib for polynomial arithmetic and
factorization.  The approach was ``brute force search'', \ie~we factorized
each polynomial, and looked for the greatest height irreducible factor.
For compactness, in the table we represent each polynomial by the list of
its coefficients.

{\small
\begin{center}
\label{tab:height1-irred}
\begin{tabular}{|r|r|l|}
\hline
     Deg   &  Height & Largest Irreducible Factor \& Height 1 Polynomial \\
\hline
       4   &     2  & $[1,2,1,1]$ Unique \\ %%% UNIQUE
           &        & $[1,1,-1,0,-1]$ \\ 
       5   &     2  & $[1,2,1,2,1]$ \\
           &        & $[1,1,-1,1,-1,-1]$ \\             %%% +- PALIN
       6   &     3  & $[1,2,3,2,2,1]$ Unique \\ %%% UNIQUE
           &        & $[1,1,1,-1,0,-1,-1]$ \\ 
       7   &     3  & $[1,2,2,3,2,2,1]$ \\
           &        & $[1,1,0,1,-1,0,-1,-1]$ \\         %%% +- PALIN
           &        & $[1,2,3,2,3,2,1]$ \\
           &        & $[1,1,1,-1,1,-1,-1,-1]$ \\        %%% +- PALIN (2nd example)
       8   &     4  & $[1,1,2,3,4,3,2,1]$ \\
           &        & $[1,0,1,1,1,-1,-1,-1,-1]$ \\ 
       9   &     4  & $[1,2,3,4,3,4,3,2,1]$ \\
           &        & $[1,1,1,1,-1,1,-1,-1,-1,-1]$ \\   %%% +- PALIN
      10   &     5  & $[1,3,4,5,5,5,4,3,1]$ \\
           &        & $[1,1,-1,0,-1,0,-1,0,-1,1,1]$ \\  %%% +-PALIN
           &        & $[1,2,4,5,5,5,4,2,1]$ \\
           &        & $[1,0,1,-1,-1,0,-1,-1,1,0,1]$ \\  %%% PALINDROMIC
      11   &     7  & $[1,2,4,6,7,7,6,4,3,1]$ \\
           &        & $[1, 0, 1, 0, -1, -1, -1, -1, 1, -1, 1, 1]$ \\ 
      12   &     9  & $[1, 3, 5, 6, 8, 9, 9, 8, 6, 3, 1]$ Unique \\  %%% UNIQUE
           &        & $[1, 1, 0, -1, 1, -1, -1, -1, -1, -1, 1, 1, 1]$ \\ 
      13   &     9  & $[1, 2, 3, 5, 6, 8, 9, 9, 8, 6, 3, 1]$ \\
           &        & $[1, 0, 0, 1, -1, 1, -1, -1, -1, -1, -1, 1, 1, 1]$ \\
      14   &    11  & $[1, 3, 6, 8, 10, 11, 11, 11, 10, 8, 6, 3, 1]$ \\
           &        & $[1, 1, 1, -1, 0, -1, -1, 0, -1, -1, 0, -1, 1, 1, 1]$ \\ %%% PALINDROMIC
      15   &    14  & $[1, 3, 6, 9, 11, 13, 14, 14, 13, 11, 8, 6, 3, 1]$ Unique \\  %% UNIQUE
           &        & $[1, 1, 1, 0, -1, 0, -1, -1, -1, -1, -1, 1, -1, 1, 1, 1]$ \\
      16   &    16  & $[1, 3, 6, 8, 11, 13, 15, 16, 16, 15, 13, 10, 6, 3, 1]$ Unique \\ %%% UNIQUE
           &        & $[1, 1, 1, -1, 1, -1, 0, -1, -1, -1, -1, -1, -1, 1, 1, 1, 1]$\\
      17   &    17  & $[1,3,6,9,12,14,16,17,16,14,12,9,6,3,1]$ \\
           &        & $[1,0,0,-1,0,-1,1,-1,-1,1,1,-1,1,0,1,0,0,-1]$ \\   %%% +-PALIN
      18   &    25  & $[1,3,6,9,13,17,21,24,25,24,22,18,13,8,4,1]$ \\
           &        & $[1,0,0,-1,1,-1,0,-1,-1,0,1,-1,1,1,1,0,1,-1,-1]$ \\
      19   &    33  & $[1, 3, 7, 12, 18, 24, 29, 32, 33, 31, 27, 22, 17, 12, 8, 4, 1]$ Unique \\  %% UNIQUE
           &        & $[1, 0, 1, -1, 0, -1, -1, -1, 0, -1, 1, 1, 1, 0, 1, -1, 1, 1, -1, -1]$ \\
      20   &    39  & $[1, 4, 8, 14, 21, 28, 34, 38, 39, 38, 35, 31, 26, 21, 15, 9, 4, 1]$ \\
           &        & $[1, 1, -1, 1, -1, -1, -1, -1, -1, 1, 0, 1, 0, 1, -1, 1, 1, 1, 0, -1, -1]$ \\
%%New max height = 39  for fac=x^17 +4*x^16 +8*x^15 +14*x^14 +21*x^13 +28*x^12 +34*x^11 +38*x^10 +39*x^9 +38*x^8 +34*x^7 +28*x^6 +21*x^5 +14*x^4 +8*x^3 +4*x^2 +2*x +1  factor of x^20 +x^19 -x^18 +x^17 -x^16 -x^15 -x^14 -x^13 -x^12 +x^11 -x^10 +x^9 +x^8 +x^7 +x^6 +x^5 -x^3 -x^2 +x -1  NOTICE THAT THE BIG FACTOR IS RATHER FAR FROM PALINDOMIC!!!
      21   &    43  & $[1, 4, 8, 14, 21, 28, 34, 39, 42, 43, 41, 37, 32, 27, 21, 15, 9, 4, 1]$ \\
           &        & $[1, 1, -1, 1, -1, -1, -1, 0, -1, 0, -1, 1, 1, 1, -1, 1, 0, 1, 1, 0, -1, -1]$ \\ %% UNIQUE
\hline
\end{tabular}
\end{center}
} %% end small

Here are a few observations on the values reported in the table.  In every
degree there was a largest factor having all positive coefficients.  Except
for degrees~$5$ and~$9$ the factors exhibit coefficients which increase
weakly to a maximum value and then decrease weakly again.  Several of these
extremal examples exhibit $\pm$-palindromic symmetry; in degrees~$7$
and~$10$ two palindromic maximal height factors were found, so we listed
both cases in the table.  In degrees~$4$, $6$, $12$, $15$, $16$, $19$
and~$21$ the maximal height factor was essentially unique (up to reversal
and mapping $x \mapsto -x$).  In degrees~$4$ to~$9$ the cofactor is $x-1$;
in degrees~$10$ to~$16$ the cofactor is $(x-1)^2$; and in degrees~$17$
to~$21$ the cofactor is $(x-1)^3$.

From this small table it appears that the heights of the largest
irreducible factor grow approximately as $0.7 \times 1.22^d$; however, since
the table is rather small, extrapolation from this data is perhaps rather
hazardous.

%%%%%%%%%%%%%%%%%%%%%%%%%%%%%%%%%%%%%%%%%%%%%%%%%%%%%%%
\subsubsection{Large Height Reducible Factors}
\label{ht1-reducible}

In this subsection we are interested in polynomials of height~$1$ having a
large factor in $\ZZ[x]$.  We note that the greatest height in
a given degree appears to grow a little faster than it did for irreducible factors.

The table below summarises the results obtained.  The entries in this table
were found using an {\it ad hoc\/} program written in C++ which relied upon
CoCoALib for polynomial arithmetic and factorization.  The approach was
``brute force search'': for each polynomial we computed all possible
products of its irreducible factors, and measured their heights.

{\small
\begin{center}
\label{tab:height1-reducible-part1}
\begin{tabular}{|r|r|l|}
\hline
     Deg   & Height & Largest Factor \& Height 1 Polynomial \\
\hline
       3   &     2  & $[1,2,1]$ \\  %% $\phi_2^2$
           &        & $[1,1,-1,-1]$ \\ 
       4   &     2  & $[1,2,1]$ \\  %% $\phi_2^2$
           &        & $[1,1,0,1,1]$ \\
       5   &     3  & $[1,2,3,2,1]$ \\  %% $\phi_3^2$
           &        & $[1,1,1,-1,-1,-1]$ \\ 
       6   &     2  & $[1,2,1]$ \\  %% $\phi_2^2$
           &        & $[1,0,-1,0,-1,0,1]$ \\
       7   &     4  & $[1,2,3,4,3,2,1]$ \\ %% $\phi_5^2$
           &        & $[1,1,1,1,-1,-1,-1,-1]$ \\
       8   &     4  & $[1,3,4,3,1]$ \\  %% $\phi_2^2 \phi_3$
           &        & $[1,0,-1,0,0,0,-1,0,1]$ \\
       9   &     7  & $[1,4,7,7,4,1]$ \\ %% $\phi_2^2 \phi_3$
           &        & $[1,1,-1,-1,0,0,-1,-1,1,1]$ \\
      10   &    10  & $[1,4,8,10,8,4,1]$ \\ %% $\phi_2^2 \phi_3^2$
           &        & $[1,1,0,-1,-1,0,-1,-1,0,1,1]$ \\
      11   &    13  & $[1,4,9,13,13,9,4,1]$ \\ %% $\phi_2 \phi_3^3$
           &        & $[1,1,1,-1,-1,-1,-1,-1,-1,1,1,1]$ \\
      12   &     9  & $[1,3,6,8,9,9,9,8,6,3,1]$ \\ %% $\phi_3^2 \phi_7$
           &        & $[1,1,1,-1,-1,-1,0,-1,-1,-1,1,1,1]$ \\
      13   &    22  & $[1,5,12,19,22,19,12,5,1]$ \\ %% $\phi_2^4 \phi_3 \phi_4$
           &        & $[1,1,-1,-1,-1,-1,0,0,1,1,1,1,-1,-1]$ \\
      14   &    17  & $[1,4,9,14,17,17,14,9,4,1]$ \\ %% $\phi_2 \phi_3^2 \phi_5$
           &        & $[1, 0, 0, -1, 0, -1, -1, 0, 1, 1, 0, 1, 0, 0, -1]$ \\
      15   &    30  & $[1,5,13,23,30,30,23,13,5,1]$ \\ %% $\phi_2^3 \phi_3^2 \phi_4$
           &        & $[1, 0, -1, -1, -1, 1, 0, 1, 1, 0, 1, -1, -1, -1, 0, 1]$ \\
      16   &    29  & $[1, 4, 9, 15, 21, 26, 29, 29, 26, 21, 15, 9, 4, 1]$ \\ %% $\phi_2 \phi_3 \phi_5 \phi_7$
           &        & $[1, 1, 0, -1, -1, -1, -1, -1, 0, 1, 1, 1, 1, 1, 0, -1, -1]$ \\
\hline
\end{tabular}
\end{center}
} %% end small

The $\pm$-palindromic symmetry present in these extreme examples is
evident, as is the fact that the factors have positive coefficients which
increase strictly towards the maximal central values.  There seems to
be no particular pattern to the cofactors.

\medskip
Beyond degree~$16$ it was impractical to conduct an exhaustive search.
Instead, given that every extreme example up to degree~$16$ found by
exhaustive search exhibited $\pm$-palindromic symmetry, we considered only
$\pm$-palindromic polynomials in higher degrees~---~this restriction greatly
reduced the number of cases to consider, and thus the computation time.  It
seems plausible to suppose that the examples found in this restricted
search are nonetheless extremal amongst all height~$1$ polynomials of that
same degree.  To keep the table compact we give only about half the
coefficients; the elided coefficients can easily be filled in because all
the polynomials are $\pm$-palindromic.  The entries in this table were obtained
using a slightly modified version of the program used for the table above.

{\small
\begin{center}
\label{tab:height1-reducible-part2}
\begin{tabular}{|r|r|l|}
\hline
     Deg   &  Height & Largest Factor \& Height 1 Polynomial\\
\hline
      17   &    42  & $[1, 5, 13, 24, 35, 42, 42, 35,\ldots]$ \\ %% $\phi_2^3 \phi_3 \phi_4 \phi_5$
           &        & $[1, 0, -1, 0, -1, -1, 0, 1, 1,\ldots, 1]$ \\
      18   &    42  & $[1, 5, 13, 24, 35, 42, 42, 35,\ldots]$ \\ %% $\phi_2^3 \phi_3 \phi_4 \phi_5$
           &        & $[1, 0, 0, 1, -1, -1, -1, -1, -1, 0,\ldots, -1]$ \\
      19   &    55  & $[1, 4, 10, 19, 30, 41, 50, 55, 55, 50,\ldots]$ \\ %% $\phi_2 \phi_3 \phi_4 \phi_5 \phi_7$
           &        & $[1, 0, 0, -1, -1, -1, 0, 0, 1, 1,\ldots, 1]$ \\
      20   &    65  & $[1, 6, 18, 36, 54, 65, 65, 54,\ldots]$ \\ %% $\phi2^3 \phi_3^2 \phi_5$
           &        & $[1, 1, 0, -1, -1, 0, -1, -1, 0, 1, 0,\ldots, -1]$ \\
      21   &   110  & $[1, 5, 14, 29, 49, 71, 91, 105, 110, 105,\ldots]$ \\ %% $\phi_2^3 \phi_3 \phi_4 \phi_5 \phi_7$
           &        & $[1, 0, -1, -1, -1, 0, 1, 1, 1, 1, 0,\ldots, -1]$ \\
      22   &   110  & $[1, 5, 14, 29, 49, 71, 91, 105, 110, 105,\ldots]$ \\  %% $\phi_2^3 \phi_3 \phi_4 \phi_5 \phi_7$
           &        & $[1, -1, -1, 0, 0, 1, 1, 0, 0, 0, -1, 0,\ldots, 1]$ \\
      23   &   161  & $[1, 7, 24, 55, 96, 136, 161, 161, 136,\ldots]$ \\ %% $\phi_2^5 \phi_3 \phi_4 \phi_5$
           &        & $[1, 1, -1, -1, -1, -1, 0, 0, 1, 1, 0, 0,\ldots, 1]$ \\
      24   &   173  & $[1, 6, 19, 42, 73, 107, 138, 161, 173, 173, 161,\ldots]$ \\  %% $\phi_2^3 \phi_3^2 \phi_4 \phi_5 \phi_8$
           &        & $[1, 0, -1, -1, 0, 0, -1, 1, 1, 1, 0, 1, 0,\ldots, -1]$ \\
      25   &   238  & $[1, 7, 25, 61, 114, 173, 220, 238, 220,\ldots]$ \\ %% $\phi_2^4 \phi_3^2 \phi_4 \phi_5$
           &        & $[1, 0, -1, -1, -1, 1, 0, 1, 1, 0, 0, -1, 0,\ldots, -1]$ \\
      26   &   233  & $[1, 6, 19, 43, 78, 120, 162, 197, 221, 233, 233, 221,\ldots]$ \\ %% $\phi_2^3 \phi_3^2 \phi_4 \phi_5 \phi_9$
           &        & $[1, 0, -1, 0, -1, -1, 0, 1, 1, 0, 1, 1, -1, 0, 1, \ldots,-1]$ \\
      27   &   356  & $[1, 7, 25, 62, 121, 197, 275, 334, 356, 334,\ldots]$ \\ %% $\phi_2^4 \phi_3 \phi_4 \phi_5^2$
           &        & $[1, 0, -1, 0, -1, -1, 0, 1, 1, 1, 0, 0, 0, -1, \ldots, -1]$ \\
      28   &   371  & $[1, 5, 15, 34, 64, 105, 155, 210, 265, 314, 351, 371, 371, 351,\ldots]$ \\ %% $\phi_2 \phi_3^2 \phi_4 \phi_5 \phi_7 \phi_9$
           &        & $[1, 0, 0, -1, -1, -1, 0, 0, 1, 0, 1, 1, 1, 1, 0,\ldots, -1]$ \\ %% $\phi_2^2 \phi_3 \phi_4 \phi_5 \phi_7 g_8$
      29   &   560  & $[1, 6, 19, 44, 84, 140, 210, 289, 370, 445, 506, 546, 560, 546,\ldots]$ \\
           &        & $[1, 1, -1, -1, -1, -1, -1, 0, 1, 1, 1, 1, 1, 1, 1,\ldots, -1]$ \\
\hline
\end{tabular}
\end{center}
} %% end small

In both of the tables above every one of the largest factors is a product
of powers of cyclotomic polynomials, with the exception of the height~$560$
polynomial which is a product of cyclotomic polynomials together with the
irreducible polynomial $g_8 = x^8+x^7+x^5+x^4+x^3+x+1$.

From the values in this second table we see that the height of the largest
factor increases roughly exponentially as $1.24^d$ where~$d$ is the
degree~---~though it is also clear that the heights do not form an
increasing sequence.  We note that the rate of growth is only slightly
greater than that for irreducible factors.

% HEIGHT 2: some results which beat height 1 (in terms of ratio):
% New max ht=9  for  2*x^4 -6*x^3 +9*x^2 -6*x +2  a factor of  2*x^7 -2*x^6 +x^5 +2*x^4 +2*x^3 +x^2 -2*x +2
% New max ht=10  for  2*x^5 -6*x^4 +10*x^3 -8*x^2 +4*x -1  a factor of  2*x^8 -2*x^7 +2*x^6 +2*x^5 +2*x^4 +x^3 -2*x^2 +2*x -1

%%%%%%%%%%%%%%%%%%%%%%%%%%%%%%%%%%%%%%%%%%%%%%%%%%%%%%%%%%%%%%%%%%%%%%%%%%%%%
\subsubsection{Height 1 multiples of $(x+1)^n$ or $(x-1)^n$}
\label{mignotte-ht1}

Here we shall see a simple way of constructing height~$1$ polynomials
having a large factor, namely $(x+1)^n$ or $(x-1)^n$.  We recall that
both $(x+1)^n$ and $(x-1)^n$ have height ${n \choose \lfloor n/2 \rfloor}
\approx 2^{n+1}/\sqrt{2\pi n}$.  For our purposes $(x+1)^n$ and $(x-1)^n$
are virtually interchangeable since if $(x+1)^n$ divides $f(x)$ then
$(x-1)^n$ divides $f(-x)$, and {\it vice versa.}

Now, for any~$n$, we construct explicitly a polynomial of height~$1$ having
$(x-1)^n$ as a factor.  Clearly $x-1$ divides $x^d-1$ for any positive
integer~$d$.  Consequently, $(x-1)^n$ divides the product $f =
\prod_{k=1}^{n} (x^{e_k}-1)$ for any choice of the exponents $e_k$.  In
general the product may have large height, but choosing the exponents
judiciously, for instance $e_k = 2^{k-1}$, gives us a polynomial of
height~$1$ (and degree $2^n-1$ in this instance).  We note also that the
cofactor $f/(x-1)^k$ is a product of polynomials of the form
$(x^{e_k}-1)/(x-1)$ all of whose coefficients are equal to~$1$.  We can
easily find a lower bound for the height of the cofactor: by considering
the degree of the cofactor, and the values at $x=1$ of each factor
$(x^{e_k}-1)/(x-1)$ we see that its height is at least $\prod e_j /\bigl(
1+\sum (e_j-1) \bigr)$.

A more general result about large factors of polynomials is Theorem~4
in~\cite{Mig88} which proves that we can find small height multiples of
polynomials in $\ZZ[x]$.  In particular, from it we can deduce that for
any~$n$ there exist height~$1$ multiples of $(x+1)^n$ of degree less than
$n^2 \log n$~---~a far lower degree than we get from the simple
construction above.  Alternatively, by writing $d=n^2 \log n$, we can
conclude that there are height~$1$ polynomials of degree at most~$d$ having
a factor of height at least $2^{\delta-1} / \sqrt{\delta}$ where $\delta
\approx \sqrt{2d/ \log d}$.

\medskip

In an earlier paper, Mignotte~\cite{Mig81} gave an (exponential) algorithm
for finding polynomials with coefficients in $\{-1, 0, 1\}$ which are
divisible by any specified power of $x+1$.  We used the algorithm to find
the lowest degree examples for each $n \le 8$: for the case $n=8$ we
found~\cite{Abb89} the lowest degree height~$1$ multiple to be $f_{41} =
(x+1)^8 \cdot g_{33}(-x)$ where
\begin{eqnarray*}
f_{41} &=& x^{41} -x^{40} - x^{39} + x^{36} +x^{35} -x^{33} + x^{32}-x^{30} -x^{27} + x^{23} + x^{22}\cr
      & & -x^{21} -x^{20} + x^{19} + x^{18} - x^{14} - x^{11} + x^9 - x^8 + x^6 + x^5 -x^2 -x +1
\end{eqnarray*}
and the cofactor is
\begin{eqnarray*}
g_{33} &=& x^{33} + 7x^{32} + 27x^{31} + 76x^{30} + 174x^{29} + 343x^{28} + 603x^{27} + 968x^{26}\cr
      & & + 1442x^{25} + 2016x^{24} + 2667x^{23} + 3359x^{22} + 4046x^{21} + 4677x^{20}\cr
      & & + 5202x^{19} + 5578x^{18} + 5774x^{17} + 5774x^{16} + 5578x^{15} + 5202x^{14}\cr
      & & + 4677x^{13} + 4046x^{12} + 3359x^{11} + 2667x^{10} + 2016x^9 + 1442x^8\cr
      & & + 968x^7 + 603x^6 + 343x^5 + 174x^4 + 76x^3 + 27x^2 + 7x + 1
\end{eqnarray*}

We observe that $f_{41}(-x)$ is indeed of the form
$\prod_{k=1}^{n}(x^{e_k}-1)$ with $n=8$ and exponents $e_k:
1,2,3,4,5,7,8,11$.  A remarkable example of this form is $f_{69}$ which has
exponents $e_k: 1,2,3,4,5,6,7,8,9,11,13$; it is a polynomial of degree~$69$
having a factor of height $2930202$.  These two polynomials along with many
other similar ones are mentioned in the article~\cite{BM00} where Borwein
and Mossinghof study specifically and in depth the question of the lowest
degree height~$1$ multiples of $(x+1)^n$.

\medskip
\noindent
{\small
{\bf Note}
We observe that certain height~$1$ polynomials of the form $\prod_{k=1}^{n}
(x^{e_k}-1)$ have the unusual property that some of the factors in the
square-free decomposition have large height; a concrete example is given by
$e_k: 3,4,5$.  We recall that one of the first steps in polynomial
factorization is to compute the square-free decomposition.
}

\section{Large Ratio Factorizations in $\ZZ[x]$}
\label{high-ratio}

In section~\ref{one-large-factor} we looked at factorizations having (at least) one
large factor, in this section we shall look at factorizations where all the
factors are large.  Our main interest is in ``short'' factorizations of the form $f =
g_1 g_2 \in \ZZ[x]$, \ie~with two (non-trivial) factors.  We concentrate our attention on
this type of factorization partly because that is conceptually the simplest
situation, and partly because that was the only case where it is feasible
to conduct thorough searches on the computer (beyond the lowest degrees).
Later on in section~\ref{many-large-irred-factors} we will see how to
construct ``longer'' examples with more than two irreducible factors.

In the Conclusion of the paper~\cite{BTW93} the authors wondered whether a
polynomial~$f \in \ZZ[x]$ must always have at least one small irreducible
factor, \ie~whose height is no greater than that of~$f$.  Collins gave a
negative answer in~\cite{Col04} where he published several quite simple
examples, including two factorizations having ratio slightly
greater than~$2$.  In this section we present many more counter-examples,
most with much larger ratio, even exceeding~$10$ in a few cases.  On the
basis of these examples it seems reasonable to conjecture that there exist
examples exhibiting arbitrarily great ratio.

The large ratio factorizations presented here also have a direct
implication for the single factor bounds of
section~\ref{single-factor-bounds} which bound the size of at least one
factor in any factorization.  Thus the example factorizations in this
section force the single factor bounds to be ``large''.

\goodbreak
%%%%%%%%%%%%%%%%%%%%%%%%%%%%%%%%%%%%%%%%%%%%%%%%%%%%%%%%%%%%%%%%%%%%%%%%%%%%%
\subsection{The  Search is Theoretically Finite}

Many of the examples we give here are probably extremal, \ie~there is no other
polynomial of the same degree having a factorization $f = g_1 g_2$ in
$\ZZ[x]$ exhibiting a greater ratio of $\min \{\height(g_1),\height(g_2)\}
/ \height(f)$.  Unfortunately, proving extremality seems to be rather difficult.
In~\cite{Rum06} Rump studied the behaviour of the two-norm under multiplication
in $\RR[x]$: in particular he looked at the minimal value $\mu = \min \{ |g_1 g_2|_2 :
|g_1|_2 = |g_2|_2 = 1 \}$ where the degrees of $g_1$ and $g_2$ are fixed.
The fact that $\RR[x]$ contains no zero-divisors together with a simple
closedness and continuity argument show that the minimum must be strictly
positive.  Rump reports that actually computing the minimum in general is
hard.  Nevertheless, knowing that $\mu > 0$ lets us prove the following lemma~---~its
practical usefulness is severely limited by the difficulty in obtaining
lower bounds for $\mu$ and by the computational cost of a truly exhaustive
search.  However, the lemma does imply the existence of attainable
rational upper bounds $\beta(d_1,d_2)$ for the ratio of short factorizations $f = g_1 g_2
\in \ZZ[x]$ with $\deg(g_1) = d_1$ and $\deg(g_2) = d_2$.

\begin{lemma}
  Let $d_1, d_2 \in \NN$ be positive.  There exists a constant $H$ depending on $d_1$
  and $d_2$ such that if $g_1, g_2 \in \ZZ[x]$ are of degree $d_1$ and
  $d_2$ respectively and $\height(g_1) > H$ or $\height(g_2) > H$ then
  $\height(g_1 g_2) > \min \{ \height(g_1), \height(g_2) \}$.  In other
  words: polynomials in $\ZZ[x]$ of fixed degree and having factorizations
  with ratio greater than~$1$ have bounded height.
\end{lemma}

\begin{proof}
  Suppose the contrary, \ie~there exists an infinite sequence of distinct pairs
  $(g_1^{(1)},g_2^{(1)}), (g_1^{(2)}, g_2^{(2)}), \ldots$ of unlimited
  height exhibiting ratio greater than~$1$, with every $\deg(g_1^{(k)})=d_1$
  and $\deg(g_2^{(k)}) = d_2$.  First note that the heights of both
  $g_1^{(k)}$ and $g_2^{(k)}$ are unlimited since otherwise there would be
  a finite limit on the height of $g_1^{(k)} g_2^{(k)}$.  But there are
  only finitely many polynomials of degree $d_1+d_2$ and bounded height,
  and unique factorization in $\ZZ[x]$ would imply only finitely many
  possible pairs.

  Thus for any $B>0$ there exists a pair $(g_1,g_2)$ with $\height(g_1) >
  B$ and $\height(g_2) > B$ and ratio $R < 1$.  Let $l_1 = |g_1|_2$ and
  $l_2 = |g_2|_2$, then we clearly also have $l_1 \ge \height(g_1) > B$ and
  $l_2 \ge \height(g_2) > B$.  Put $\bar g_1 = g_1/l_1$ and $\bar g_2 = g_2/l_2$, so
  $|\bar g_1|_2 = |\bar g_2|_2 = 1$.  Considering their product, we see that
\begin{eqnarray*}
|\bar g_1 \bar g_2|_2 &\le& \sqrt{d_1+d_2} \,\height(\bar g_1 \bar g_2) \cr
            & = & \sqrt{d_1+d_2} \,\height(g_1 g_2)/l_1l_2 \cr
            & < & \sqrt{d_1+d_2} \, R \min \{ \height(g_1), \height(g_2)\}/l_1l_2 \cr
            & < & \sqrt{d_1+d_2} \, R/l_s \cr
            & < & \sqrt{d_1+d_2} / B \cr
\end{eqnarray*}
where $s$ is the index of whichever of $g_1$ and $g_2$ has greater height.
This contradicts the strict positivity of $\mu$.
\end{proof}

%%%%%%%%%%%%%%%%%%%%%%%%%%%%%%%%%%%%%%%%%%%%%%%%%%%%%%%%%%%%%%%%%%%%%%%%%%%%%
\subsection{A Family of Irreducible Factorizations of Ratio $> 1$}

Here is a simple result showing that there are irreducible factorizations
with ratio $> 1$ in every even degree (from~$10$ upwards)~---~the result
depends on a conjecture (see below).

Let $n>1$ be an integer, and let $g_1 = nx^2-(2n-1)x+n \in \ZZ[x]$ and $g_2
= (x^{n+1}+x^n+x^{n-1}+\cdots+1)(x^{n+2}+x^{n+1}+x^n+\cdots+1) \in \ZZ[x]$.
Note that $\height(g_2) = n+2$.  It is not hard to show that the product
$g_1 g_2$ has height $n+1$; similarly, it can be shown that both $g_1 (g_2 x+1)$
and $g_1 (g_2 x^3+x^2+x+1)$ also have height $n+1$.  Now $g_2$ is obviously
reducible, but we make the following conjecture~---~verified for all $n < 1000$.

\subsubsection*{Conjecture}
(with the notation of the preceding paragraph)
\begin{itemize}
\item $g_2x+1$ is irreducible whenever $n \not\equiv 2 \bmod 3$,
\item $g_2x^3+x^2+x+1$ is irreducible whenever $n \equiv 2 \bmod 3$.
\end{itemize}

\bigskip
\noindent
{\small
{\bf Note} In practice, for small values of~$n$ it seems to be easy to find irreducible
polynomials $\tilde g_2$ of height $2n-1$ for which the product $g_1 \tilde g_2$ has
height~$n$, though we have not managed to discern any obvious pattern in such
polynomials.  We do note that the minimal degree of suitable $\tilde g_2$ does
rise quickly as~$n$ increases: starting from $n=2$ the first few minimal
degrees are $7,11,14,20,22,26$.
}

%%%%%%%%%%%%%%%%%%%%%%%%%%%%%%%%%%%%%%%%%%%%%%%%%%%%%%%%%%%%%%%%%%%%%%%%%%%%%
\subsection{A Family of Factorizations with Unlimited Ratio}
\label{symmetric-reducible-family}

Here we present an explicit family where the ratio increases exponentially
with degree.  Though we are primarily interested in irreducible
factorizations, the bounds in section~\ref{single-factor-bounds} apply to
any factorization; so this family constrains those bounds to grow rapidly.

Let $f(x) = (x+1)(x^2+x+1) = x^3+2x^2+2x+1$.  Then $1-x^6 = f(x) \cdot
f(-x)$ is a $*$-symmetric factorization.  We obtain our family simply by
taking powers: namely, $(1-x^6)^k = f^k(x) \cdot f^k(-x)$.  We estimate the
height of $(1-x^6)^k$ using the binomial theorem and Stirling's
approximation: $\height \bigl( (1-x^6)^k\bigr) \approx 2^{k+1}/\sqrt{2 \pi
  k}$.  We do not have an explicit formula for the asymptotic height of $f^k$~---~but see
Conjecture~$1$ in the Conclusion.  Nevertheless,
by considering the value of $f(x)$ at $x=1$, we can
deduce that $\height(f^k) \ge 6^k/(3k+1)$.  Hence the ratio of the
factorization is greater than $3^k/(3k+1)$.  From this we obtain $1.2^d$ as
an asymptotic lower bound for the greatest ratio among (reducible)
factorizations of polynomials of degree~$d$.

\medskip
\noindent
{\small
{\bf Note}  An exhaustive search in low degrees suggests that this family of
  factorizations is close to extremal.
}

\subsection{Extremal Short Factorizations}
\label{low-deg-2-large-factors}

In this section we present the results of some extensive computer searches
for large ratio short factorizations in $\ZZ[x]$.  The searches were conducted
degree by degree, and in each degree we tried all possibilities for the
degrees of the two factors.  As we commented earlier it is hard to be
absolutely certain that these examples are extremal; however, if there are
other factorizations exhibiting a greater ratio then the factors must have
considerably larger height than the example given in the table.

The examples in the table are not generally unique (\ie~in most cases there
are other polynomials of the same degree whose factorizations exhibit the
same ratio).  Up to and including degree~$14$ the examples are almost
certainly extremal; we think the examples in degrees~$15$, $16$ and $17$
are probably extremal, but doubt that the example in degree~$18$ is
extremal.  The computations were done using an {\it ad hoc\/} program in C++.

{\small
\begin{center}
\label{tab:pairs-reducible}
\begin{tabular}{|r|r|l|}
\hline
     Deg   &  Ratio & Example factorization\\
\hline
       5   &   2    & $[1, 2, 1]\times[1, -2, 2, -1]$ \\
       6   &   2    & $[1, 2, 2, 1]\times[1, -2, 2, -1]$ \\
       8   &   4    & $[1, 3, 4, 3, 1]\times[1, -3, 4, -3, 1]$ \\
       9   &   4    & $[1, 3, 4, 3, 1]\times[1, -3, 4, -4, 3, -1]$ \\
      10   &   5    & $[1, 3, 5, 5, 3, 1]\times[1, -3, 5, -5, 3, -1]$ \\
      11   &   7    & $[1, 4, 7, 7, 4, 1]\times[1, -4, 8, -10, 8, -4, 1]$ \\
      12   &   7    & $[1, 4, 7, 7, 4, 1]\times[1, -4, 8, -11, 11, -8, 4, -1]$ \\
      13   &  11    & $[1, 5, 11, 14, 11, 5, 1]\times[1, -4, 8, -11, 11, -8, 4, -1]$ \\
      14   &  12    & $[2, 9, 19, 24, 19, 9, 2]\times[1, -5, 13, -22, 26, -22, 13, -5, 1]$ \\  %% NOT 7+7!!!
\hline
      15   &  18    & $[1, 5, 12, 18, 18, 12, 5, 1]\times[1, -5, 12, -19, 22, -19, 12, -5, 1]$ \\
      16   &  18    & $[1, 6, 17, 30, 36, 30, 17, 6, 1]\times[2, -10, 25, -41, 48, -41, 25, -10, 2]$ \\  %% NOT MONIC!!!
      17   &  24    & $[1, 6, 18, 35, 48, 48, 35, 18, 6, 1]\times[2, -10, 25, -41, 48, -41, 25, -10, 2]$ \\  %% INCOMPLETE SEARCH!!!
      18   &  23    & $[1, 5, 12, 19, 23, 23, 19, 12, 5, 1]\times[1, -5, 12, -19, 23, -23, 19, -12, 5, -1]$ \\ %% less than 24!!!
\hline
\end{tabular}
\end{center}
} %% end small

At the higher degrees we limited the searches heuristically.  Beyond
degree~$13$ we considered only pairs whose degrees differ by at most~$2$,
because all the best examples found in lower degrees had this property.  Up
to degree~$14$ it is very evident that there is always a best example whose
factors are palindromic, thus beyond degree~$14$ we searched only for pairs
of palindromic polynomials.

It is also clear that the factorizations are of the form $g_1(x) \cdot
g_2(-x)$ where $g_1$ and $g_2$ have positive coefficients which increase
strictly towards the central terms~---~the same phenomenon we observed in
section~\ref{ht1-reducible}.  We note the curious fact that the examples in
degrees~$16$ and~$17$ contain the same non-monic factor.

%%%%%%%%%%%%%%%%%%%%%%%%%%%%%%%%%%%%%%%%%%%%%%%%%%%%%%%%%%%%%%%%%%%%%%%%%%%%%
\subsection{Extremal Irreducible Short Factorizations}
\label{irred-large-ratio}

Here we look at factorizations $f = g_1 g_2 \in \ZZ[x]$ with the additional
requirement that the two factors be irreducible.  We will see that this
extra requirement limits severely the growth of the ratio compared to what
we observed in the previous subsection.  Ideally, for the purposes of
polynomial factorization, we would like to have a single factor bound which
is valid only for irreducible factors in $\ZZ[x]$  The examples here show
that a single factor bound specific to irreducible factors could be far
smaller than the ones we currently know: from this table we see that the
log of the ratio grows roughly as $0.07\, d - 0.3$ where $d$ is the degree.

The computations were done using an {\it ad hoc\/} program in C++ which
conducted a ``weak irreducibility'' test (\ie~which simply tested for
divisibility by a few commonly occurring polynomials, primarily low index
cyclotomics).  The final verification of irreducibility was effected using
CoCoA~\cite{CoCoA}.

\begin{center}
\label{tab:pairs-irred}
\begin{tabular}{|r|r|l|}
\hline
      Deg  &  Ratio & Example factorization\\
\hline
       7   &   1.25 & $[3, 5, 3]\times[-1, 3, -5, 4, -3, 1]$ \\
       8   &   1.50 & $[2, 4, 6, 5, 2]\times[2, -5, 6, -4, 2]$ \\
       9   &   1.50 & $[2, 3, 2]\times[1, -1, 1, 0, -2, 3, -2, 1]$ \\
           &        & $[1, 2, 3, 3, 2]\times[1, -2, 3, -1, -1, 1]$ \\
      10   &   1.50 & $[1, 1, 0, -2, -3, -2]\times[2, -3, 2, 0, -1, 1]$ \\
      11   &   1.67 & $[2, 4, 5, 4, 2]\times[1, -1, 0, 2, -4, 5, -3, 1]$ \\
      12   &   2.00 & $[1, 1, 1, 0, -1, -2, -1]\times[-1, 2, -1, 0, 1, -1, 1]$ \\
\hline
      14   &   2.00 & $[1, 1, 0, -1, 0, 1, 2, 1]$ \\ %%nothing better with height <= 20
      16   &   2.16 & $[2, 6, 8, 3, -6, -13, -12, -6, -1]$ \\
      18   &   2.71 & $[4, 12, 19, 18, 9, -2, -7, -6, -3, -1]$ \\
      20   &   3.00 & $[1, 1, 0, -1, -1, -1, 0, 2, 3, 2, 1]$ \\
\hline
      22   &   3.50 & $[2, 6, 11, 14, 13, 7, 0, -5, -8, -9, -6, -2]$ \\ %%(unique, ~17hrs on pitagora) ht <= 17, ratio >3.01
      24   &   4.25 & $[2, 6, 10, 11, 8, 0, -10, -17, -16, -8, 1, 4, 2]$ \\ %%(unique, about 36hrs on abbottpb) ht<=19, ratio >=4
%      26   &   3.66 & $[1, 4, 7, 7, 3, -3, -9, -11, -8, -1, 5, 7, 4, 1]$ \\ %%(unique soln, ~19hrs on pitagora) ht<=11, ratio >3.6
%      28   &   4.00 & $[1, 2, 3, 4, 5, 4, 1, -4, -8, -8, -4, 0, 2, 2, 1]$ \\ %%(3 solns, ~17hrs in pitagora)
%      32   &   4.00 & $[1, 3, 4, 3, 1, -1, -3, -4, -4, -2, 1, 3, 2, 0, -2, -2, -1]$ \\ Lowest deg with ratio=4 & ht=4
%      32   &   4.00 & $[1, 1, 1, 1, 2, 2, 1, -2, -4, -4, -2, -1, 0, 1, 2, 2, 1]$ \\
\hline
\end{tabular}
\end{center}

The examples given in the table are not generally unique (\ie~in most cases
there are other pairs of polynomials exhibiting the same ratio).  Up to and
including degree~$12$ the examples are very probably extremal, \ie~any
example exhibiting a greater ratio must have factors of considerably larger
height.  It is quite possible that in degrees~$22$ and higher there are
examples of slightly greater height exhibiting greater ratios; the
necessary computations were not attempted because they would take
unreasonably long.

It is very evident that in the even degrees (up to and including degree~$12$)
there is always a $*$-symmetric best example (\ie~where $g_1(x) = \bar g_2(-x)$).
To save time in the searches beyond degree~$12$, we restricted ourselves to
even degrees and to $*$-symmetric factorizations.

\medskip

Note the example in degree~$16$ exhibits a greater ratio than Collins's
rather larger ``winning'' polynomial (on page~$1518$ in~\cite{Col04}); on
the following page he claims an example of the same degree with ratio about~$2.20$, but we
suspect that there is a misprint, and that that polynomial is the same as
our example in degree~$20$ given in the next subsection.  Our example in
degree~$20$ is much smaller than Collins's examples and exhibits a
significantly larger ratio, namely~$3.00$.

%%%%%%%%%%%%%%%%%%%%%%%%%%%%%%%%%%%%%%%%%%%%%%%%%%%%%%%%%%%%%%%%%%%%%%%%%%%%%
\subsection{Extremal Palindromic $*$-Symmetric Factorizations}
\label{palindromic-irred}

In the previous subsection the exponential increase in candidate
factorizations made it impractical to conduct full searches at higher
degrees.  Inspired by the appearance of palindromicity and $*$-symmetry in
earlier examples, and realising that these symmetries greatly reduce the
number of candidates to consider, we decided to restrict attention to
palindromic $*$-symmetric factorizations.  While this will not let us find
the probable extreme ratio in each degree it does at least establish a
lower bound for that value.

Since any odd degree palindromic polynomial is always divisible by $x+1$,
and since we are seeking palindromic $*$-symmetric irreducible short
factorizations, we consider only degrees which are multiples of~$4$.  The
search was effected using an {\it ad hoc\/} C++ program which conducted a
brute force search with some simple pruning criteria; irreducibility of the
factors was subsequently verified using CoCoA~\cite{CoCoA}.  Here is the
resulting table:

{\small
\begin{center}
\label{tab:palindromic-irred}
\begin{tabular}{|r|r|l|}
\hline
     Deg   &  Ratio & Corresponding $g_1$\\
\hline
       8   &   1.25 & $[2,4,5,4,2]$ \\
      12   &   1.48 & $[5,17,31,37,31,17,5]$ \\
      16   &   1.69 & $[13,63,157,256,299,256,157,63,13]$ \\
      20   &   2.20 & $[29,175,543,1119,1683,1921,1683,1119,543,175,29]$ \\
      24   &   3.28 & $[4,14,22,13,-17,-53,-69,-53,-17,13,22,14,4]$ \\
      28   &   4.25 & $[2,6,10,10,4,-6,-14,-17,-14,-6,4,10,10,6,2]$ \\
      32   &   7.22 & $[3,13,29,41,37,12,-24,-54,-65,\ldots,3]$ \\
      36   &  11.37 & $[6,34,98,182,234,194,41,-181,-377,-455,\ldots,6]$ \\
      40   &  13.75 & $[6,33,93,175,243,249,158,-26,-248,-427,-495,\ldots,6]$ \\
\hline
\end{tabular}
\end{center}
} %% end small

As in the other subsections here, it is difficult to be completely certain of extremality;
nevertheless we believe the examples up to and including degree~$32$ to be extremal among
palindromic $*$-symmetric factorizations since any example exhibiting greater ratio must
have factors of considerably larger height.  The example in degree~$36$ is
probably extremal too, but the search was deliberately limited to where we
expected to find a good example.  The example in degree~$40$ is unlikely to
be extremal: several {\it ad hoc\/} tricks were used to make the search
faster (\ie~less slow).

We note that from degree~$24$ the best example in each degree contains
negative coefficients, and has much smaller height compared to the best
example in degree~$20$.  Collins in~\cite{Col04} chose to consider only
positive coefficients and reported finding no example with ratio greater
than $2.20$.  We suspect that the factorization in degree~$20$ is the
example hinted at on page~$1519$ in~\cite{Col04}, though he reported
(apparently mistakenly) that it was in degree~$16$.  We point out that our
degree~$20$ example in the previous subsection has a larger ratio.

%%%%%%%%%%%%%%%%%%%%%%%%%%%%%%%%%%%%%%%%%%%%%%%%%%%%%%%%%%%%%%%%%%%%%%%%%%%%%
\subsection{Height 1 Polynomials with 2 Large Irreducible Factors}
\label{ht1-large-ratio}

After looking at the examples in the preceding subsection one might think
it is necessary to consider polynomials with ever larger coefficients in
order to find large ratio (short) factorizations.  Here we
see that apparently it is enough to consider polynomials of height~$1$.
While the few examples here and in the next subsection certainly do not
prove that we can go on indefinitely, they seem to be enough to let us
conjecture there are polynomials of height~$1$ whose irreducible factorizations
exhibit arbitrarily great ratios.

These examples (and those in the following subsection) are particularly
interesting because, together with the result in
section~\ref{many-large-irred-factors}, they show that there are
polynomials of height~$1$ with arbitrarily many irreducible factors of
large height (at least up to~$11$).

%% See Pairs-ht1
Here is a small table of factorizations of polynomials of height~$1$ which
exhibit increasing ratios.  The first three examples are surely of minimal
degree; we are not sure whether the others are of lowest possible degree
exhibiting that ratio.  The $*$-symmetric examples with ratios~$4$ and~$5$
are of minimal degree amongst $*$-symmetric factorizations with those
ratios.  The search was effected using an {\it ad hoc\/} C++ program which
conducted a brute force search with some simple pruning criteria;
irreducibility of the factors was subsequently verified using
CoCoA~\cite{CoCoA}.

{\small
\begin{center}
\label{tab:palindromic-irred-ht1}
\begin{tabular}{|r|r|l|}
\hline
   Ratio  &    Deg  & Factorization\\
\hline
       2  &     12  & $[1, 2, 1, 0, -1, -1, -1] *\hbox{-symmetric}$ \\
       3  &     20  & $[1, 2, 3, 2, 0, -1, -1, -1, 0, 1, 1] *\hbox{-symmetric}$ \\
       4  &     27  & $[1, 1, 0, -2, -4, -4, -2, 0, 2, 3, 3, 2, 1] \times$ \\ % just two solns
          &         & $[1, -2, 2, 0, -1, 0, 1, 0, -1, 1, 0, 0, -2, 4, -3, 1]$ \\
          &     32  & $[1, 1, 1, 1, 2, 2, 1, -2, -4, -4, -2, -1, 0, 1, 2, 2, 1] *\hbox{-symmetric}$ \\
       5  &     30  & $[1, 1, -1, -4, -5, -4, -1, 1, 1]\times$ \\
          &         & $[1, -2, 2, -1, -1, 3, -4, 4, -2, -1, 4, -5, 4, -1, -2, 4, -4, 3, -1, -1, 2, -2, 1]$ \\
       5  &     32  & $[1, 2, 2, 2, 2, 0, -3, -5, -5, -4, -2, 0, 2, 3, 3, 2, 1] *\hbox{-symmetric}$ \\ %% LOWEST DEG of form f(x)*f(-1/x)
       6  &     40  & $[1, 2, 2, 0, -2, -3, -1, 2, 4, 4, 3, 1, -2, -5, -6, -4, -1, 1, 2, 2, 1] *\hbox{-symmetric}$ \\
\hline
\end{tabular}
\end{center}
} %% end small

%%%%%%%%%%%%%%%%%%%%%%%%%%%%%%%%%%%%%%%%%%%%%%%%%%%%%%%%%%%%%%%%%%%%%%%%%%%%%
\subsection{Height 1 Polynomials with Large Ratio Factorizations:\\ palindromic $*$-symmetric factors}
\label{ht1-large-ratio-palin}
%% See PseudoSquares-ht1

Extending the searches of the previous subsection to ratios greater than~$6$
would have been prohibitively expensive.  Somewhat arbitrarily we decided
to restrict consideration to palindromic $*$-symmetric factorizations in an
attempt to find examples with greater ratio.

The examples here suggest that for any height~$H$ there exist polynomials
of height~$1$ having irreducible $*$-symmetric palindromic factorizations
of ratio least~$H$.  The table below exhibits for each height one of the
factors; for compactness we have used ellipsis for the higher degree
factors.  Up to and including height~$6$ the polynomials are surely of
minimal degree; for heights~$7$ and greater we restricted the search space,
and so could conceivably have missed some lower degree example.  We note
that the degrees are not always increasing.  There appears to be no obvious
pattern.

{\small
\begin{center}
\label{tab:palindromic-ht1-irred}
\begin{tabular}{|r|r|l|}
\hline
   Ht  &    Deg  & Palindromic factor\\
\hline
    2   &     32  & $[1,1,1,1,1,0,-1,-2,-2,-2,-1,0,1,1,1,1,1]$ \\
    3   &     36  & $[1,1,1,1,1,0,0,-1,-2,-3,-2,-1,0,0,1,1,1,1,1]$ \\
    4   &     68  & $[1,3,4,3,1,-1,-3,-4,-3,-1,1,2,2,1,0,-1,-1,-1,\ldots]$ \\
    5   &     76  & $[1,2,2,1,0,0,1,2,2,1,0,-1,-1,0,1,1,0,-2,-4,-5,\ldots]$ \\
    6   &     72  & $[1,2,2,1,0,-1,-2,-3,-3,-3,-3,-3,-2,-1,0,2,5,6,6,\ldots]$ \\
\hline
    7   &    100  & $[1,3,5,6,5,2,-1,-2,-1,1,3,3,1,-2,-5,-7,-7,-5,-2,$ \\
        &         & $ \quad 0,0,-1,-1,0,2,4,5,5,\ldots]$ \\
    8   &     84  & $[1,3,5,5,2,-3,-7,-8,-6,-2,2,3,1,-2,-4,-4,-1,3,5,4,2,1,\ldots]$ \\
    9   &     88  & $[1,3,5,6,6,5,3,0,-3,-5,-6,-7,-8,-9,-9,-7,-3,1,4,6,7,7,7,\ldots]$ \\
   10   &    120  & $[1,3,4,3,1,-1,-3,-5,-7,-8,-7,-5,-3,-1,1,3,5,7,9,10,9,6,3,$ \\
        &         & $ \quad 0,-3,-5,-5,-4,-3,-3,-3,\ldots]$ \\
   11   &    100  & $[1, 3, 5, 5, 2, -3, -7, -8, -6, -2, 3, 7, 8, 5, 0, -4, -5, -3, 1,$ \\
        &         & $ \quad 5, 7, 6, 2, -4, -9, -11,\ldots]$ \\
\hline
\end{tabular}
\end{center}
} %% end small

%%%%%%%%%%%%%%%%%%%%%%%%%%%%%%%%%%%%%%%%%%%%%%%%%%%%%%%%%%%%%%%%%%%%%%%%%%%%%
\subsection{Large Ratio Factorizations of $x^d-1$ in $\RR[x]$ and $\CC[x]$}
\label{xd-1-large-ratio}

In this subsection we extend briefly our horizon to encompass
factorizations in $\RR[x]$ and $\CC[x]$.  The main reason for doing so is
that the single factor bounds are valid for any factorization in $\CC[x]$
satisfying the scaling hypothesis.  We shall see that this wide applicability of
the bounds forces them to much larger than necessary for bounding factors
in $\ZZ[x]$.

In section~\ref{large-factors-xd-1} we saw that the apparently innocuous
polynomial $x^d-1$ can have quite large factors in $\RR[x]$ or $\CC[x]$.
Here we shall see that it also admits large ratio factorizations, at least
when $d$ is even; as a consequence the single factor bounds must grow
exponentially with degree.  Below we shall use the specific primitive
$d$-th root of unity $\zeta=\exp(i \theta_d)$ where $\theta_d = 2 \pi /d$.

First we consider the case $d=4t+2$.  Here we find that a greatest
height factor is $u_d = \prod_{k=-t}^{t} (x-\zeta^k)$, \ie~the polynomial whose roots are all $d$-th
roots of unity with strictly positive real part.  Moreover $u_d \in
\RR[x]$ is palindromic and $x^d-1 = u_d(x) \cdot u_d(-x) = u_d(x)
\cdot u_d^*(x)$, that is we have a $*$-symmetric factorization in
$\RR[x]$.  Since the factorization is $*$-symmetric, its ratio is
equal to the height of $u_d$.

Empirically, based on computed values up to $d=100$, we observe that
$\height(u_d) \approx 0.097 \times 1.34^n$, or equivalently $\log
\height(u_d) \approx 0.29 d - 2.33$.  For a degree $\lfloor d/2 \rfloor$
factor the best bounds are the binomial bound (with $\rho=1$) or Mignotte's
bound (with $M(f)=1$): they allow height up to ${\lfloor d/2 \rfloor
  \choose \lfloor d/4 \rfloor}$ which we can approximate using Stirling's
formula to obtain $2^{1+d/2}/\sqrt{\pi d} \sim 1.41^d/\sqrt{d}$.  Thus
we see that the bounds are too large by only about 20\% in logarithmic terms.
%% a bit tighter than we found in section~\ref{xd-1-large-RR-factor}.

We do not have a proof of correctness of the empirical formula above, but
we can prove a lower bound on the height of $u_d$.  We use the
factorization $u_d(x) = (x-1) \prod_{k=1}^t (x^2 -2 \cos (k \theta_d) x
+1)$.  Since $u_d$ has degree $2d+1$ it must have height at least $|
u_d(-1)| /(2d+2)$.  We estimate $|u_d(-1)|$ using the factorization
$u_d(-1) = -2 \prod_{k=1}^t (2+2 \cos (k \theta_d))$.  Observing that for
each index $k$ we have $\cos (k \theta_d) + \cos ((t+1-k) \theta_d) > 1$,
it is easy to show that $\prod_{k=1}^t (1+\cos (k \theta_d)) > 2^{t/2}$.
Hence $|u_d(-1)| \ge 2^{3t/2+1} = 2^{(3d+2)/8}$.  And we can thus conclude
that asymptotically $\log \height(u_d) > 0.2599 d$.

\medskip

For the case $d=4t+4$, we obtain a largest ratio $*$-symmetric
factorization $x^d-1 = v_d(x) \cdot v_d(-x) \in \CC[x]$ where $v_d(x) =
u_d(x) \cdot (x+i)$.  Empirically, we observe that $\log \height(v_d)
\approx 0.29 d - 2.33$ which is the same growth rate seen above for the
factor $u_d$.  The asymptotic lower bound continues to be valid in this
case too.

%%%%%%%%%%%%%%%%%%%%%%%%%%%%%%%%%%%%%%%%%%%%%%%%%%%%%%%%%%%%%%%%%%%%%%%%%%%%%
%%%%%%%%%%%%%%%%%%%%%%%%%%%%%%%%%%%%%%%%%%%%%%%%%%%%%%%%%%%%%%%%%%%%%%%%%%%%%
\section{Longer Large Ratio Irreducible Factorizations}
\label{many-large-irred-factors}

In section~\ref{high-ratio} we concentrated on short factorizations
(\ie~with just two factors), and it is natural to ask whether there are
polynomials having more than two irreducible factors each of which has
greater height than the original polynomial.

Direct search is limited by its exponential nature.  We did not succeed in
finding in a reasonable time any small examples (\ie~with three irreducible
factors each of low degree and low height).  Extending the search looks to
be hopeless.  Nevertheless we believe that fairly small examples exist, but
we just do not know a good way to find them.  Below we present a
construction which allows us to create high degree examples with an even
number of irreducible factors; currently we know of no similar method for
constructing examples with an odd number of irreducible factors.

\subsection{A Handy Lemma}

\begin{lemma}\label{handy-lemma}
Let $f \in \ZZ[x]$ be irreducible and not cyclotomic.  Then we have the following:
\begin{enumerate}
\item $f(x^p)$ is reducible in $\ZZ[x]$ for only finitely many primes $p$

\item if $p$ is an odd prime, and both $f(x)$ and $f(x^p)$ are
irreducible then $f(x^{p^k})$ is irreducible for all $k \in\NN$.

\item if $p$ is an odd prime then $f(x^p)$ is irreducible if $f(0)/\lc(f)$ is not a $p$-th power in $\QQ$.
\end{enumerate}
\end{lemma}

\begin{proof}
  Let $K:\QQ$ be a splitting field extension for $f$.  So we have a
  factorization $f(x) = \prod (x-\alpha_i) \in K$ where the $\alpha_i$ are
  the roots of $f$ in $K$.  Let $p$ be a prime; then obviously we have a
  factorization $f(x^p) = \prod (x^p-\alpha_i)$ in $K[x]$.  Now suppose
  that $f(x^p)$ is reducible in $\ZZ[x]$ and let $g \in \ZZ[x]$ be one of
  its irreducible factors.  Then each polynomial $x^p-\alpha_i$ has a
  non-trivial factor $\gcd(g(x), x^p-\alpha_i) \in K[x]$.  But for any odd
  prime $p$, by Capelli's theorem, $x^p-\alpha_i$ is reducible if and only
  if $\alpha_i$ is a $p$-th power (or of the form $-4\beta^4$ if $p=2$).
  Since $f$ is not cyclotomic, its roots are not roots of unity, and thus
  $f(x^p)$ is reducible only for finitely many primes $p$.

  For the second claim, when both $f(x)$ and $f(x^p)$ are irreducible we know
  from Capelli's Theorem that the $\alpha_i$ are not $p$-th powers.  The
  irreducibility of $x^{p^k}-\alpha_i$ immediately follows, and hence $f(x^{p^k})$
  is irreducible too.

  For the third claim, the norm of each root $\alpha_i$ is just $f(0)/\lc(f)$.
  If $\alpha_i$ is a $p$-th power then so must its norm be; conversely, if the
  norm is not a $p$-th power then $\alpha_i$ cannot be, and $f(x^p)$ is irreducible.
\smallskip
  I am indebted to Barry Trager for the outline of this proof.
\end{proof}

%% Conj: if f(x^m) & f(x^n) are irred then so is f(x^{mn}) provided m and n are of the
%% form 4^a*odd

\subsection{The Construction of Longer Factorizations}

Starting from some of the high ratio short factorizations in
section~\ref{high-ratio}, we now show how to construct polynomials having
at least~$4$ irreducible factors each of height greater than the original
polynomial.  We suppose we already have an irreducible factorization
$f=g_1g_2$ where $\height(f)^2 < \min(\height(g_1), \height(g_2))$~---~note
that the height of $f$ is squared.  Any example from
section~\ref{ht1-large-ratio} is a suitable candidate; for instance, the
first example gives us a height~$1$ polynomial with irreducible factors
$g_1 = x^6-x^5+x^4-x^2+2x-1$ and $g_2 = x^6 + 2x^5 + x^4 - x^2 - x - 1$.

Now we look for a power $k$ such that $\height \bigl( f(x) \cdot f(x^k)
\bigr) = \height(f)^2$ and both $g_1(x^k)$ and $g_2(x^k)$ are irreducible.
Clearly, choosing $k > \deg(f)$ automatically satisfies the first
condition~---~some smaller values may also happen to work.
Lemma~\ref{handy-lemma} tells us that there are infinitely many values of $k$
which satisfy the irreducibility conditions.  Thus we have infinitely many
choices for $k$ which satisfy all the conditions.  In our specific example
we can take $k=13$.  This leads us to $f(x) \cdot f(x^{13})$ a polynomial
of degree~$168$ and height~$1$ having four irreducible factors each of
height~$2$.  Here we see plainly that this construction produces
polynomials of rather large degree.

Now, since $\height(f)=1$ we can repeat the above process to produce a
polynomial of height~$1$ having any specified even number of irreducible
factors each of height~$2$.  For instance, $f(x) \cdot f(x^{13}) \cdot
f(x^{169})$ is a polynomial of degree~$28560$ and height~$1$ having six
irreducible factors each of height~$2$.

\medskip

Here is another example of the same construction.  We take the palindromic
$*$-symmetric example in degree~$28$ from section~\ref{palindromic-irred}:
specifically we take
\begin{eqnarray*}
  g_1 &=&  2x^{14} + 6x^{13} + 10x^{12} + 10x^{11} + 4x^{10} - 6x^9 - 14x^8 - 17x^7\cr
      & & - 14x^6 - 6x^5 + 4x^4 + 10x^3 + 10x^2 + 6x + 2
\end{eqnarray*}
and $g_2(x) = g_1(-x)$, thus we have that $\height(g_1 g_2)^2 <
\height(g_1)$.  By experimentation we find that the exponent $k=15$
satisfies all the conditions.  So the product $g_1(x) g_2(-x) g_1(x^{15})
g_2(-x^{15})$ is a polynomial of degree~$448$ and height~$16$ having four
irreducible factors each of height~$17$.  However, this time we cannot
repeat the process because $\height(g_1 g_2)^3 > \height(g_1)$.

%%%%%%%%%%%%%%%%%%%%%%%%%%%%%%%%%%%%%%%%%%%%%%%%%%%%%%%%%%%%%%%%%%%%%%%%%%%%%
%%%%%%%%%%%%%%%%%%%%%%%%%%%%%%%%%%%%%%%%%%%%%%%%%%%%%%%%%%%%%%%%%%%%%%%%%%%%%
\section {Conclusion}

In this paper we have compared various factor coefficient bounds, and shown
with concrete examples that no one bound is universally better or worse
than the others~---~so it is always a good idea to compute them all and
pick whichever comes out smallest.  We have refined some existing bounds,
and have made explicit a bound latent in an article by Mignotte.  We have
shown how the degree aware bounds can be combined to produce a better bound
in some instances.

In the second part of the paper we have exhibited several examples of
factorizations with unusually large factors which explain partly why the
factor bounds have to be rather larger than one might expect.
Nevertheless, the known bounds are almost always ``far too big'', most
especially if we are concerned with the sizes of irreducible factors in
$\ZZ[x]$.  Indeed, all the degree aware bounds remain valid for (suitably
scaled, reducible) factors in $\CC[x]$, and this seems to be the main
reason why the bounds are so loose.  If good bounds can be found for
irreducible factors in $\ZZ[x]$ then we expect a consequent significant
improvement in speed when factorizing polynomials in $\ZZ[x]$.

Currently, in the context of practical polynomial factorization, the
problem of overly large bounds is mitigated by using an ``engineering''
technique known as {\it early termination\/} during the lifting phase; we
feel that better bounds would furnish a more ``mathematical'' response to
the same problem.

\medskip

Collins~\cite{Col04} had already published some examples (including some
very small ones) answering the question in the Conclusion of~\cite{BTW93}.
Here we have extended considerably the set of known examples, and have given
examples exhibiting far larger ratio than those given by Collins.  We have also
improved the extremal ratio in degree~$16$ found by Collins.  Our examples
in section~\ref{palindromic-irred} boost credence in the conjecture that
there exist factorizations with arbitrarily large ratios (even if we
restrict to palindromic $*$-symmetric factorizations of height~$1$
polynomials), though they fall short of providing a proof or a concrete
family.

In section~\ref{symmetric-reducible-family} we exhibited a family of
reducible $*$-symmetric factorizations with unbounded ratio, where the
ratio grows exponentially with degree.  This family thus forces any single
factor bound to grow at least as fast.

The examples in section~\ref{irred-large-ratio} compel any ideal
``irreducible single factor bound'' to grow with degree, though the rate of
growth appears to be much slower than for single factor bounds valid for
any (suitably scaled) factorization in $\CC[x]$.  This suggests that such
an ideal single factor bound could be very much smaller than the currently
known ones.

We have also shown how to construct polynomials with any even number of
irreducible factors and whose irreducible factorization has ratio greater
than~$1$.  Using the examples from section~\ref{ht1-large-ratio-palin} we
can create explicitly examples with ratio up to~$11$.  So far we have not
encountered any irreducible factorization with an odd number of factors and
ratio greater than~$1$, though we think it very likely that such
factorizations exist.

\subsection*{Questions and Conjectures}

There are a number of unanswered questions related to the results contained
in this article.  Here are a few of them:
\begin{small}
\begin{itemize}
\item Do two factor irreducible factorizations of arbitrarily great ratio
  exist?
\item How does the maximal ratio of two factor factorizations
  increase with degree?
\item Is there an easy way to construct high ratio two factor
  factorizations?
\item Do high ratio factorizations with an odd number of factors exist?
\end{itemize}
\end{small}

We conclude with two more conjectures which arose during our studies.  The
first cropped up while conducting the studies presented in
section~\ref{symmetric-reducible-family}.  The second one is the
consequence of some fruitless searches for polynomials whose square has
lower height than the original~---~this search was inspired by the existence
of high ratio palindromic $*$-symmetric factorizations which ``look vaguely
like the square of a polynomial''.  The second conjecture has a natural
extension to higher powers, but this extended form is backed up by far less
computational evidence.

\subsubsection*{Conjecture 1}

Let $f \in \CC[x]$ then asymptotically $\height(f^n) \approx K_f
|f|^n_\circ / \sqrt{n}$ where $K_f$ is a constant depending on~$f$, and
$|f|_\circ = \max \{ f(e^{2\pi i \theta}) : 0 \le \theta < 1 \}$ is the
maximum modulus of~$f$ on the unit circle in $\CC$.

\bigskip
\begin{small}
\noindent
 {\bf Note} In the particular case of $f=x+1$ we can use the
  binomial theorem and Stirling's approximation to deduce that $K_f =
  \sqrt{2/\pi}$.  For the general case, F.~Amoroso kindly pointed out that it is
  not hard to show that
$$
|f|^n_\circ (1+n \deg f)^{-1} \le \height(f^n) \le |f|^n_\circ
$$
\end{small}

\subsubsection*{Conjecture 2}

Let $f \in \ZZ[x]$ be non-zero and not of the form $\pm x^d$.  Then
$\height(f^2) \ge 2 \height(f)$.

\bigskip
\begin{small}
\noindent
{\bf Note} The factor~$2$ in the conjecture cannot be replaced by a larger value
since we have equality for $(x+1)^2 = x^2+2x+1$.  Indeed there are many
other polynomials which achieve equality.  For instance, taking
$f=x^{14}-x^{12}-x^{10}-x^8-4x^7+x^6+x^4+x^2-1$ we obtain a ``non-trivial''
example satisfying $\height(f^2) = 2 \height(f)$.  The largest height example with
this property which we have found is a polynomial of height~$30$ and
degree~$1680$.  It seems likely that there are examples of arbitrarily great height.
\end{small}

\subsubsection*{Conjecture 2 (extended)}

Let $f \in \ZZ[x]$ be non-zero and not of the form $\pm x^d$.  Then for any
integer $k>0$ we have $\height(f^k) \ge R_k \height(f)$ where the factor
$R_k = {k \choose \lfloor k/2 \rfloor}$.

\bigskip
\begin{small}
\noindent
{\bf Note} Choosing $f = x+1$ shows that the factor $R_k$ cannot be made
larger.  Equality can also be achieved for other polynomials: \eg~with $k=3$ we
can take $f=x^5 + x^4 - x + 1$.  Indeed the only polynomials we have
found which achieve equality are (obviously) all binomials of height~$1$, and certain
quadrinomials of height~$1$.
\end{small}

\end{document}